\documentclass[11pt]{amsart}
\usepackage{amsmath, amssymb, amsfonts, amsthm, verbatim, dsfont, graphicx, enumerate, pifont,tikz,tikz-cd}

\usepackage{hyperref}

\usepackage[backend=bibtex]{biblatex}

\theoremstyle{theorem}
\newtheorem{theorem}{Theorem}[section]
\newtheorem{lemma}[theorem]{Lemma}
\newtheorem{proposition}[theorem]{Proposition}
\newtheorem{corollary}[theorem]{Corollary}
\newtheorem{conjecture}[theorem]{Conjecture}

\theoremstyle{definition}
\newtheorem{definition}[theorem]{Definition}
\newtheorem{remark}[theorem]{Remark}

\newtheorem{example}[theorem]{Example}

\providecommand{\abs}[1]{\left\lvert#1\right\rvert}

\providecommand{\OP}[1]{\operatorname{#1}}

\begin{document}
\author{Igor Uljarevic}
\title{Floer homology of automorphisms of Liouville domains}
\maketitle
\begin{center}\today \end{center}
\begin{abstract}
We introduce a combination of fixed point Floer homology and symplectic homology for Liouville domains. As an application, we detect non-trivial elements in the symplectic mapping class group of a Liouville domain.
\end{abstract}

\newcommand{\A}{\mathcal{A}}
\renewcommand{\P}{\mathcal{P}}
\newcommand{\M}{\mathcal{M}}
\newcommand{\N}{\mathcal{N}}
\newcommand{\T}{\tau} % fibered Dehn twist
\newcommand{\D}{\operatorname{Diff}}
\newcommand{\irm}{\OP{irm}} % the ideal restriction map
\newcommand{\slospec}{\mathcal{S}}
\renewcommand{\hat}{\widehat}
\newcommand{\hcl}{\theta}

\section{Introduction}

The aim of this paper is to introduce Floer homology for symplectomorphism of a Liouville domain. The Floer homology group is associated to a Liouville domain $(W,\lambda)$, an exact symplectomorphism $\phi:W\to W$ (see Definition~\ref{def:exactsymp}), that is equal to the identity near the boundary, and an element $-\infty<a\leqslant \infty,$ called the slope. One should think of the groups $HF_\ast(\phi,a)$ as an extension to the Liouville domains of the Floer homology groups in~\cite{MR1297130}. On the other hand, these groups also generalize symplectic homology \cite{MR1254813,MR1306027,MR1312580,MR1408861,MR1726235}, in that
\[SH(W;\mathbb{Z}_2)=HF(\OP{id,\infty}).\]
The construction of our Floer Homology groups relies on a new variant of the standard action functional,  which is adapted to the setting of twisted loops associated to exact symplectomorphisms (see section~\ref{subsec:action}).
The Floer homology groups $HF_\ast(\phi,a)$ are invariant under compactly supported isotopy of $\phi.$ This makes them an effective tool for studying the symplectic mapping class group, i.e. the group of isotopy classes of symplectomorphisms of $(W,d\lambda)$ which are equal to the identity near the boundary. An important class of such symplectomorphisms are so called fibered Dehn twists of Liouville domains with periodic Reeb flow on the boundary (see Definition~\ref{FibDTwist} below).

\begin{theorem}\label{thm:SlopePeriodicity}
Let $(W,\lambda)$ be a Liouville domain with connected boundary such that the Reeb  flow of $(\partial W,\left.\lambda\right|_{\partial W})$ is 1-periodic. Let $a\in\mathbb{R}$ be a real number that is not a period of a periodic Reeb orbit of $(\partial W,\left.\lambda\right|_{\partial W}).$ If the fibered Dehn twist represents a class of order $\ell\in \mathbb{N}$ in $\pi_0\OP{Symp}_c(W,d\lambda),$ then
\begin{equation}\label{eq:SlopePeriodicity}
HF(\OP{id},a)\cong HF(\OP{id},a+\ell).
\end{equation}
\end{theorem}
In fact, we will prove a more general result, which provides necessary conditions for a loop of contactomorphisms of $\partial W$ to induce an element in $\pi_0\OP{Symp}_c(W,d\lambda)$ of certain (prescribed) order (see Theorem~\ref{thm:SlopePeriodicityGen}).
\begin{corollary}\label{thm:DTInfOrd}
Let $(W,\lambda)$ be as in Theorem~\ref{thm:SlopePeriodicity}. If
\begin{equation}\label{eq:dim<dim}
\dim H(W;\mathbb{Z}_2)< \dim SH(W;\mathbb{Z}_2),
\end{equation}
then the fibered Dehn twist represents a class of infinite order in $\pi_0\OP{Symp}_c(W,d\lambda).$ Here, $\dim H(W;\mathbb{Z}_2)$ stands for the sum of Betti numbers rather than the dimension of the homology group of a particular degree.  
\end{corollary} 

Special cases of Corollary~\ref{thm:DTInfOrd} include the squares of the Dehn-Seidel twists on $T^\ast\mathbb{S}^n$ and their extensions to the cotangent bundles of other symmetric spaces (see Corollary 4.5 in \cite{MR1765826}).

\begin{definition}
The homology $H_\ast(X;\mathbb{Z}_2)$ of a topological space $X$ is said to be \textbf{symmetric} if there exists $k\in \mathbb{Z}$ such that
\[H_{k-j}(X;\mathbb{Z}_2)\cong H_j(X;\mathbb{Z}_2)\]
for all $j\in\mathbb{Z}.$
\end{definition}

\begin{corollary}\label{cor:DTNonTrivC0}
Let $(W,\lambda)$ be as in Theorem~\ref{thm:SlopePeriodicity}. Assume that the Reeb flow of $(\partial W, \left.\lambda\right|_{\partial W})$ induces a free circle action on $\partial W,$  and that the first Chern class $c_1(W)$ of $(W,d\lambda)$ vanishes. If the homology $H_\ast(W;\mathbb{Z}_2)$ is not symmetric, then the fibered Dehn twist represents a nontrivial class in $\pi_0\OP{Symp}_c(W,d\lambda).$
\end{corollary}
The corollary applies when $W$ is a smooth degree $d\geqslant 2$ projective hypersurface in $\mathbb{C}P^m, m>3$ with a neighbourhood of a smooth hyperplane section removed (see Example~\ref{HyperplaneSection}). We will also prove a version of the corollary without the assumption on $c_1$ (see Corollary~\ref{thm:DTNonTriv}).

For $a>0$ small, the group $HF_\ast(\phi,a)$ has already appeared in the work of Seidel \cite{MR1882336} and  McLean \cite{MR2960024}.
After the first version of this paper was written, the paper~\cite{Seidel2014} by Seidel was posted on the ArXiv. That paper has an overlap with the current paper in that it independently introduces a version of fixed point Floer homology for exact symplectomorphisms that is very similar to the theory developed in the present paper.

The paper is organized as follows. Section~2 introduces the general set-up and defines the Floer homology groups. Section~3 is devoted to Theorem~\ref{thm:SlopePeriodicity} and its applications. Section~4 defines the Floer homology for more general slopes and proves an extension of Theorem~\ref{thm:SlopePeriodicity} to a more general setting (Theorem~\ref{thm:SlopePeriodicityGen}).

\subsection{Notation and conventions}\label{sec:Conv}
Throughout, $(W,\lambda)$ is a $2n$-dimensional Liouville domain with connected boundary $M:=\partial W,$ symplectic form $\omega:=d\lambda,$ and induced contact form $\beta:=\left.\lambda\right|_{M}$ on the boundary. The associated Reeb vector field is denoted by $R^\beta,$ the Reeb flow by $\sigma^\beta_t,$ and the contact distribution by $\xi^\beta.$
By definition of Liouville domain, the Liouville vector field $X_\lambda$ (defined by $X_\lambda\lrcorner d\lambda=\lambda $) points out on the boundary. It can be used to identify a neighbourhood of the boundary with the product $M\times (0,1]$ with the Liouville form $r\beta$ (where $r$ is real coordinate).

The completion $\hat{W}$ is obtained by 
\[\hat{W}:=W\sqcup M\times (0,\infty)/\sim,\]
where $\sim$ stands for the above mentioned identification. With a slight abuse of notation, the Liouville form on $\hat{W}$ is also denoted by $\lambda.$ We identify $W$ and $M\times (0,\infty)$ with the appropriate regions in $\hat{W}.$ We say that a statement is true in $\hat{W}$ ``near infinity'' if it holds in the complement of a compact subset.

A function $H:\hat{W}\to \mathbb{R}$ and its Hamiltonian vector field $X_H$ are related by $dH= \omega\left(X_H,\cdot\right).$
The Hamiltonian isotopy associated to a time-dependent Hamiltonian function $H:\mathbb{R}\times\hat{W}\to \mathbb{R}$ will be denoted by $\psi_t^H.$ It is defined by $\partial_t \psi_t^H= X_{H_t}\circ \psi^H_t, \psi^H_0=\OP{id},$ where $H_t:=H(t,\cdot).$ An almost complex structure $J$ on $\hat{W}$ is $\omega$-compatible if $\omega(\cdot,J\cdot)$ is a Riemannian metric.

We denote by $\OP{Symp}_c(W,\omega)$ the group of symplectomorphisms of $W$ compactly supported in the interior of $W,$ and by $\OP{Cont}(M,\xi^\beta)$ the group of co-orientation preserving contactmomorphisms of $M.$
\subsection*{Acknowledgments}
This work will be part of my PhD thesis. I would like to thank Paul Biran and Dietmar Salamon for their guidance and for their enormous help in improving the exposition of this paper. I am grateful to the referee for reading the paper carefully and for suggesting that $HF(\OP{id},a)$ might be enough for the main applications. This work was supported by the Swiss National Science Foundation (grant number 200021\_134747).

\section{Floer homology for symplectomorphism of a Liouville domain}\label{sec:HF}
\subsection{Floer data and Admissible data}

\begin{definition}\label{def:exactsymp}
A symplectomorphism $\phi:\hat{W}\to\hat{W}$ is called \textbf{exact} if the 1-form $\phi^\ast\lambda-\lambda$ is exact and compactly supported. Under these assumptions there exists a unique compactly supported function $F_\phi:\hat{W}\to \mathbb{R}$ such that 
\[\phi^\ast\lambda-\lambda=dF_\phi.\]
The exact symplectomorphisms form a group, denoted by $\OP{Symp}(\hat{W},\lambda/d ).$ The functions associated to composition and inverse are given by
\begin{equation}
\begin{split}
F_{\phi_0\circ\phi_1}&=F_{\phi_0}\circ\phi_1 +F_{\phi_1},\\
F_{\phi^{-1}}&=-F_{\phi}\circ \phi^{-1}.
\end{split}
\end{equation}
We define the subgroup $\OP{Symp}_c(W,\lambda/d)$ of $\OP{Symp}(\hat{W},\lambda/d)$ by
\begin{equation*}
\OP{Symp}_c(W,\lambda/d):=\left\lbrace \phi\in \OP{Symp}(\hat{W},\lambda/d)\:|\: \OP{supp}\phi\subset W-M\right\rbrace.
\end{equation*}
\end{definition}
An exact symplectomorphism is not required to be compactly supported.
Every Hamiltonian symplectomorphism generated by a compactly supported Hamiltonian is exact. However, an exact symplectomorphism need not be Hamiltonian or even isotopic to the identity.

\begin{definition}\label{def:slospec} 
A real number $a$ is called \textbf{admissible} if it is not a period of a Reeb orbit of $(M,\beta)$. 
\end{definition}

\begin{definition}
Let $\phi\in\OP{Symp}_c(W,\lambda/d)$. \textbf{Floer data} for $\phi$ is a pair $(H,J)$ of a (time-dependent) Hamiltonian $H_t:\hat{W}\to\mathbb{R}$ and a family $J_t$ of $\omega$-compatible almost complex structures on $\hat{W}$ satisfying the following conditions. $H$ and $J$ are twisted by $\phi,$ i.e.
\begin{align}
H_{t+1}&=H_t\circ\phi,\label{HamTwist}\\
J_{t+1}&=\phi^\ast J_t. \label{JTwist}
\end{align}
In addition, the conditions below hold near infinity
\begin{align}
&H_t(x,r)= ar,\label{HamLin}\\
&J_t(x,r)\xi^\beta=\xi^\beta,\label{JCondInfComp}\\
&J_t(x,r)\partial_r=R^\beta,\label{JCondInfDel}
\end{align}  
where $a\in\mathbb{R}$ is admissible. If we want to specify the slope, we say ``$(H,J)$ is Floer data for $(\phi,a).$''
\end{definition}
\begin{remark}\label{rmk:Compatibility}
Condition~\eqref{JCondInfComp} together with $\omega$-compatibility of $J_t$ implies that $\left.J_t\right|_{\xi^\beta}$ is a compatible complex structure on the symplectic vector bundle $(\xi^\beta,d\beta)\to M.$
\end{remark}
\begin{definition}\label{def:RegFlDat}
Let $\phi\in \OP{Symp}_c(W,\lambda/d)$. Floer data $(H,J)$ for $\phi$ is \textbf{regular} if $H$ is nondegenerate with respect to $\phi,$ i.e.
\[\det\left( d(\phi\circ\psi^H_1)(x)-\OP{id}\right)\not=0,\]
for all fixed points $x$ of $\phi\circ\psi^H_1,$ and if the linearized operator for every solution of the Floer equation~\eqref{FloerEq} is surjective.
\end{definition}

\subsection{Action functional}\label{subsec:action}
\begin{definition}\label{def:action}
Let $\phi:\hat{W}\to\hat{W}$ be an exact symplectomorphism and let $H_t:\hat{W}\to\mathbb{R}$ be a Hamiltonian satisfying \eqref{HamTwist}. The \textbf{action functional} $\A_{\phi,H}$ is a function defined on the \textbf{twisted loop space }
\[\Omega_\phi:=\{\gamma:\mathbb{R}\to \widehat{W}\:|\:\phi(\gamma(t+1))=\gamma(t)\}\]
by
\[\A_{\phi,H}(\gamma):=-\int_0^1\left(\gamma^\ast\lambda +H_t(\gamma(t))dt\right)-F_\phi(\gamma(1)).\]
\end{definition}
\begin{lemma}
Let $\phi$ and $H_t$ be as in Definition~\ref{def:action}. The critical points of the action functional $\A_{\phi,H}$ are Hamiltonian twisted loops, i.e. the elements of the set
\[\P_{\phi,H}:=\left\lbrace \gamma\in \Omega_\phi\:|\:\dot{\gamma}=X_{H_t}\circ\gamma \right\rbrace.\]
\end{lemma}
\begin{proof}
Let $\gamma\in\Omega_\phi$ and let $\zeta\in T_\gamma\Omega_\phi.$ This means $\zeta$ is a section of the vector bundle $\gamma^\ast T\hat{W}$ such that
\[d\phi(\zeta(t+1))=\zeta(t).\]
The derivative of $\A_{\phi,H}$ at the point $\gamma$ in the direction $\zeta$ is given by
\[d\A_{\phi,H}(\gamma)\zeta=\left.\frac{d}{ds}\right|_{s=0}\A_{\phi,H}(\gamma_s),\]
where $\gamma_s$ is a smooth family of twisted loops satisfying
\[\left.\frac{d}{ds}\right|_{s=0}\gamma_s(t)=\zeta(t).\]
Since
\[\left.\frac{d}{ds}\right|_{s=0}\left(\gamma_s^\ast\lambda\right)= d\lambda\left( \zeta(t),\dot{\gamma}(t)\right)dt+ \frac{d}{dt}(\lambda(\zeta(t)))dt,\]
we get
\begin{equation}\label{eq:critpt}
\begin{split}
d\A_{\phi,H}(\gamma)\zeta&= -\int_0^1\left( \omega(\zeta(t),\dot{\gamma}(t)) + \frac{d}{dt}(\lambda(\zeta(t))) + dH_t(\zeta(t))\right)dt -dF_\phi(\zeta(1))\\
&=-\int_0^1 \omega\left(\zeta(t),\dot{\gamma}(t)-X_{H_t}(\gamma(t))\right)dt -\lambda(\zeta(1))+ \lambda(\zeta(0))- \phi^\ast\lambda(\zeta(1))+\lambda(\zeta(1))\\
&=\int_0^1\omega\left(\dot{\gamma}(t)-X_{H_t}(\gamma(t)),\zeta(t)\right)dt.
\end{split}
\end{equation}
In the last equality, we used the identity
\[d\phi(\zeta(1))=\zeta(0).\]
Since \eqref{eq:critpt} holds for all $\zeta\in T_\gamma\Omega_\phi,$ we have $d\A_{\phi,H}(\gamma)=0$ if and only if  $\dot{\gamma}(t)=X_{H_t}(\gamma(t))$ for all $t.$
\end{proof}

If $J_t$ is a family of almost complex structures on $\hat{W},$ that is $\omega$-compatible and satisfies~\eqref{JTwist}, then we can define a Riemannian metric on $\Omega_\phi$ by
\[\left\langle \xi,\zeta\right\rangle :=\int_0^1 \omega(\xi(t),J_t\circ\gamma(t)\zeta(t))dt,\quad\gamma\in\Omega_\phi,\:\xi,\zeta\in T_\gamma\Omega_\phi.\]

The negative gradient flow lines of $\A_{\phi,H}$ with respect to this inner product are solutions $u:\mathbb{R}^2\to\widehat{W}$ of the Floer equation
\begin{equation}\label{FloerEq}
\partial_s u+ J_t(u)\left( \partial_t u-X_{H_t}(u) \right)=0.
\end{equation}
satisfying the following periodicity condition
\begin{equation}\label{eq:percond}
\phi\circ u(s,t+1)= u(s,t).
\end{equation}
\subsection{Grading}
\begin{definition}\label{RMasl}
Let $\phi:\hat{W}\to\hat{W}$ be a symplectomorphism. Let $H^i,i=0,1$ be two Hamiltonians nondegenerate with respect to $\phi$ and satisfying~\eqref{HamTwist}. And let  $\gamma_i\in\P(\phi,H^i),i=0,1$ be twisted loops representing the same class in $\pi_0\Omega_\phi.$ Consider homotopy
\[[0,1]\times \mathbb{R}\to \hat{W}\]
between them through $\Omega_\phi.$ Write it as a map
\[u:\mathbb{R}\times\mathbb{R}\to \hat{W}\]
such that
\begin{align}
u(s,t)&=\phi\circ u(s,t+1),\\
u(s,t)&=\left\lbrace\begin{matrix}\gamma_0(t)& \text{for }s\leqslant 0\\ \gamma_1(t)& \text{for }s\geqslant 1. \end{matrix} \right.
\end{align}
We choose a symplectic trivialization $T_{s,t}:\mathbb{R}^{2n} \to T_{u(s,t)}\widehat{W}$ satisfying
\[\phi_\ast\circ T_{s,t+1}=T_{s, t}.\]
The \textbf{relative Maslov index} of $u$ is the number
\[\mu(u)=\mu_{CZ}(\Psi^1)-\mu_{CZ}(\Psi^0)\in\mathbb{Z},\]
where 
\[\Psi^i:[0,1]\to Sp(2n)\::\:t\mapsto T^{-1}_{i,t}\circ d\psi^{H^i}_t(\gamma_i(0)) T_{i,0},\quad\text{for }i=0,1. \]
It does not depend on the choice of trivialization.
\end{definition}
We need the notion of a mapping torus in order to define the relative Maslov index for the pair $(\gamma_0,\gamma_1)$ (independent of the choice of homotopy $u$).
\begin{definition}
Let $\phi:\hat{W}\to\hat{W}$ be a symplectomorphism. The \textbf{mapping torus} of $\phi$ is the fibration $p_\phi:\hat{W}_\phi\to\mathbb{S}^1$ defined by
\[\hat{W}_\phi:=\frac{\widehat{W}\times \mathbb{R}}{(\phi(x),t)\sim (x,t+1)},\quad p_\phi([x,t]):=[t]\in\mathbb{R}/\mathbb{Z}.\]
\end{definition}
Denote the vertical tangent bundle of $\hat{W}_\phi$ by
\[T^V\hat{W}:=\ker dp_\phi\to \hat{W}_\phi.\]
The symplectic form $\omega$ on $\hat{W}$ makes $T^V\hat{W}$ into a symplectic vector bundle.
A loop in $\Omega_\phi,$ seen as a map
\[u:\mathbb{R}/\mathbb{Z}\times\mathbb{R}\to\hat{W}\]
satisfying $\phi\circ u(s,t+1)=u(s,t),$ can be identified with the map
\[\mathbb{R}/\mathbb{Z}\times\mathbb{R}/\mathbb{Z} \to\hat{W}_\phi\::\:(s,t)\mapsto [(u(s,t),t)].\]
Let $\hcl\in\pi_0\Omega_\phi.$ Associated to it is the \textbf{minimal Chern number} $N_\hcl.$ It is defined as the minimal positive generator of the group
\[N_\hcl \mathbb{Z}:=\left\langle c_1(T^V\hat{W}),\pi_1(\Omega_\phi,\hcl)\right\rangle\subset \mathbb{Z}. \label{eq:ChernNum}\]
We define the minimal positive generator of trivial group to be $\infty,$ and $\mathbb{Z}_\infty:= \mathbb{Z}.$

\begin{definition}
The \textbf{relative Maslov index} for a pair $(\gamma_0,\gamma_1)\in\P(\phi,H^0) \times \P(\phi,H^1)$ of twisted loops representing the same class $\hcl\in\pi_0\Omega_\phi$ is the number
\[\mu(\gamma_0,\gamma_1):= \mu(u)\pmod{2N_\hcl}\in \mathbb{Z}_{2N_\hcl},\]
where $u,H^0,H^1$ are as in Definition~ \ref{RMasl}. 
\end{definition}
\subsection{Taming infinity}
\begin{definition}\label{def:ContData}
Let $\phi\in\OP{Symp}_c(W,\lambda/d),$ and let $(H^-,J^-)$ and $(H^+,J^+)$ be regular Floer data for $\phi.$ \textbf{Continuation data} from $(H^-,J^-)$ to $(H^+,J^+)$ consists of Hamiltonians $H_{s,t}, G_{s,t}: \hat{W}\to\mathbb{R}$ and a family of almost complex structures $J_{s,t}$ on $\hat{W}$ such that the following conditions are satisfied. For each $s\in\mathbb{R},$ $H_s,G_s$ and $J_s$ are twisted by $\phi,$ i.e.
\[H_{s,t}\circ\phi=H_{s,t+1},\quad G_{s,t}\circ\phi=G_{s,t+1},\quad \phi^\ast J_{s,t}=J_{s,t+1}.\] In addition, there exists $r_0\in (0,\infty)$ such that $G_{s,t}(x,r)=0, H_{s,t}(x,r)=a(s,t)$ and the conditions~\eqref{JCondInfComp}  and \eqref{JCondInfDel} hold for $J_s(x,r),$ for all $s,t\in\mathbb{R}, r\geqslant r_0, x\in M$ and some smooth function $a:\mathbb{R}^2\to \mathbb{R}$ that is increasing with respect to $s$ (i.e. $\partial_s a(s,t)\geqslant 0$). Finally, 
\[(H_{s,t},G_{s,t},J_{s,t})=(H_{t}^\pm,0,J_{t}^\pm),\quad\text{for }\pm s>>0.\]
\end{definition}
\begin{remark}
We can find continuation data from $(H^-,J^-)$ to $(H^+,J^+)$ whenever the slope of $H^-$ is less or equal to the slope of $H^+.$
\end{remark}
\begin{definition}
Let $\phi, H_{s,t}, G_{s,t}$ and $J_{s,t}$ be as in Definition~\ref{def:ContData}. The \textbf{energy} of a solution $u:\mathbb{R}^2\to\widehat{W}$ of the $s$-dependent (generalized) Floer equation
\begin{equation}\label{FloerEqS} 
\partial_s u-X_{G_{s,t}}(u)+ J_{s,t}(u)\left( \partial_t u-X_{H_{s,t}}(u) \right)=0
\end{equation} 
satisfying periodicity condition~\eqref{eq:percond} is defined to be
\[E(u):=\int_{-\infty} ^\infty \int_0^1 \abs{\partial_s u-X_{G_{s,t}}(u)}_J^2dtds.\]
\end{definition}
\begin{lemma}\label{EId}
Let $\phi, H^\pm, J^\pm, H_s,G_s,J_s$ be as in Definition~\ref{def:ContData}, let $\gamma^\pm\in \P(\phi,H^\pm)$ and let $u:\mathbb{R}^2\to \widehat{W}$ be a solution of \eqref{FloerEqS}, \eqref{eq:percond} such that $\lim_{s\to\pm\infty} u(s,t)= \gamma^\pm(t).$ Then
\begin{equation}
\begin{split}
E(u)=&\A_{\phi,H^-}(\gamma^-)- \A_{\phi,H^+}(\gamma^+)-\int_{-\infty} ^{\infty} \int_0^1 (\partial_sH_{s,t})(u)dtds\\
&+ \int_{-\infty} ^{\infty} \int_0^1 (\partial_tG_{s,t})(u)dtds + \int_{-\infty} ^{\infty} \int_0^1 \{G_{s,t}, H_{s,t}\}(u)dtds.
\end{split}
\end{equation}
Here, $\{G_{s,t}, H_{s,t}\}:= \omega(X_{G_{s,t}}, X_{H_{s,t}})$ stands for the Poisson bracket.
\end{lemma}
\begin{proof}
\begin{equation}\label{EIdEq:1}
\begin{split}
E(u)=& \int_{-\infty}^{+\infty} \int_{0}^{1} \omega\left(\partial_s u- X_{G_{s,t}}(u),J_{s,t}(u)(\partial_s u-X_{G_{s,t}})\right)dtds\\
=&\int_{-\infty}^{+\infty} \int_{0}^{1} \omega\left(\partial_s u- X_{G_{s,t}}(u),\partial_t u-X_{H_{s,t}}(u)\right)dtds\\
=& \int_{-\infty}^{+\infty} \int_{0}^{1} \omega\left(\partial_s u,\partial_t u\right)dtds+ \int_{-\infty}^{+\infty} \int_{0}^{1} dH_{s,t}(\partial_s u)dtds -\\
& - \int_{-\infty}^{+\infty} \int_{0}^{1} dG_{s,t}(\partial_t u)dtds + \int_{-\infty}^{+\infty} \int_{0}^{1} \{G_{s,t}, H_{s,t}\}(u)dtds.
\end{split}
\end{equation}
We denote the first three terms in the last line of~\eqref{EIdEq:1} by $A$, $B$, and $C$, respectively. By applying the Stokes theorem and using $\phi\circ u(s,1)=u(s,0),$ we get
\begin{equation}
\begin{split}
A&= \int_0^1\left( \gamma^+\right)^\ast \lambda- \int_{-\infty}^{+\infty} \lambda(\partial_s u(s,1))ds- \int_0^1\left( \gamma^-\right)^\ast \lambda + \int_{-\infty}^{+\infty} \lambda(\partial_s u(s,0))ds\\
&=\int_0^1\left( \gamma^+ \right)^\ast \lambda- \int_0^1\left( \gamma^-\right)^\ast \lambda + \int_{-\infty}^{+\infty} \left(\phi^\ast\lambda- \lambda \right)(\partial_su(s,1))ds\\
&= \int_0^1\left( \gamma^+ \right)^\ast \lambda- \int_0^1\left( \gamma^-\right)^\ast \lambda + F_\phi\circ\gamma^+(1)- F_\phi\circ\gamma^-(1),
\end{split}
\end{equation} 
\begin{equation}
\begin{split}
B&= \int_{-\infty}^{+\infty} \int_0^1 \partial_s \left(H_{s,t}(u)\right) - (\partial_sH_{s,t})(u) dtds\\
&=\int_0^1H_{s,t}\circ \gamma^+(t)dt - \int_0^1H_{s,t}\circ \gamma^-(t)dt- \int_{-\infty}^{+\infty} \int_0^1 (\partial_sH_{s,t})(u) dtds.
\end{split}
\end{equation}
\begin{equation}
\begin{split}
C=& \int_{-\infty}^{+\infty} \int_0^1 \partial_t(G_{s,t}(u))-\partial_t G_{s,t}(u) dtds\\
=& \int_{-\infty}^{+\infty} G_{s,1}(u(s,1))- G_{s,0}(u(s,0)) ds - \int_{-\infty}^{+\infty} \int_0^1\partial_t G_{s,t}(u) dtds\\
=& -\int_{-\infty}^{+\infty} \int_0^1\partial_t G_{s,t}(u) dtds.
\end{split}
\end{equation}
The lemma follows by substituting $A$, $B$, and $C$ in \eqref{EIdEq:1}.
\end{proof}
\begin{lemma}\label{thm:maxprinclin}
Let $H_{s,t}, G_{s,t}, J_{s,t}$ and $r_0$ be as in Definition~\ref{def:ContData}, and let $u:U\to M\times (r_0,\infty)$ be a solution of \eqref{FloerEqS}, where $U\subset\mathbb{R}^2$ is open and connected. Then $r\circ u$ has no local maxima unless it is constant.
\end{lemma}
\begin{proof}
See page~\pageref{pf:maxprinclin}.
\end{proof}

\subsection{Chain complex}
\begin{theorem}\label{Moduli}
Let $\phi\in\OP{Symp}_c(W,\lambda/d),$ let $(H,J)$ be regular Floer data for $\phi$ and let $\gamma^-,\gamma^+\in \P(\phi,H).$ Then the set 
\[\M(\phi,H,J,\gamma^-,\gamma^+):= \left\lbrace u:\mathbb{R}^2\to\hat{W}\:|\:\eqref{FloerEq}\:\&\: \eqref{eq:percond}\:\& \lim_{s\to\pm\infty}u(s,t)= \gamma^\pm(t) \right\rbrace\]
is a manifold (possibly empty or with connected components of varying dimensions). The connected component containing $u_0$ has dimension equal to $\mu(u_0).$ The quotient of the $k-$dimensional part $\M^k(\cdots)$ of $\M(\cdots)$ by the natural $\mathbb{R}-$action is denoted by $\widehat{\M}^k(\cdots).$ 
\end{theorem}
\begin{proof}
See page~\pageref{proof:Moduli}.
\end{proof}
\begin{theorem}\label{Moduli1}
In the situation of Theorem~\ref{Moduli}, the set $\widehat{\M}^1(\phi,H,J,\gamma^-,\gamma^+)$ is finite. Its cardinality modulo 2 is denoted by $n(\phi,H,J,\gamma^-,\gamma^+).$
\end{theorem}
\begin{proof}
See page~\pageref{proof:Moduli}.
\end{proof}
\begin{theorem}\label{Moduli2}
Let $\phi\in\OP{Symp}_c(W,\lambda/d),$ let $(H,J)$ be regular Floer data for $\phi$, let $\hcl\in \pi_0\Omega_\phi,$ and let $\gamma_0\in \P(\phi,H)\cap \hcl=:\P^\hcl(\phi,H).$  Then
\[\left(CF_\ast(\phi,H,J,\hcl),\partial_{H,J}\right)\]
is a $\mathbb{Z}_{2N_\hcl}$-graded chain complex (see page~\pageref{eq:ChernNum} for  the definition of $N_\theta$), where
\[CF_k(\phi,H,J,\hcl):=\bigoplus_{ \begin{matrix} \gamma\in \P^\hcl(\phi,H)\\ \mu(\gamma,\gamma_0)=k \end{matrix}} \mathbb{Z}_2\left\langle \gamma\right\rangle,\]
and $\partial = \partial_{H,J}$ is defined by
\[\partial \left\langle \gamma \right\rangle := \sum_{\begin{matrix} \gamma'\in \P^\hcl(\phi,H)\\ \mu(\gamma,\gamma')=1 \end{matrix}} n(\phi,H,J,\gamma,\gamma')\left\langle \gamma'\right\rangle. \]
We denote the homology of this complex by
\[HF_\ast(\phi,H,J,\hcl).\]
Note that the grading depends on the choice of $\gamma_0.$
\end{theorem}
\begin{proof}
See page~\pageref{proof:Moduli}.
\end{proof}
\subsection{Continuation maps}
\begin{theorem}\label{ModuliS}
Let $\phi\in\OP{Symp}_c(W,\lambda/d),$ let $(H^\alpha,J^\alpha)$ and $(H^\beta, J^\beta)$ be regular Floer data for $\phi,$ and let $\gamma^\alpha\in\P(\phi,H^\alpha), \gamma^\beta\in\P(\phi, H^\beta).$ Then, for generic continuation data $(H_s, G_s, J_s)$ from $(H^\alpha, J^\alpha)$ to $(H^\beta, J^\beta),$ the set
\[\begin{split}
\M(\phi, \{H_s,G_s,J_s\}, \gamma^\alpha,\gamma^\beta):= \bigg\lbrace u:\mathbb{R}^2\to\hat{W}\:|\:&\eqref{FloerEqS}\:\&\:\eqref{eq:percond}\:\&\\
&\lim_{s\to-\infty}u(s,t)=\gamma^\alpha(t)\:\&\\
&\lim_{s\to+\infty}u(s,t)=\gamma^\beta(t)\bigg\rbrace
\end{split}\]
is a manifold (cut out transversely by the Floer equation). The connected component containing $u_0$ has dimension $\mu(u_0).$ The $k$-dimensional part of $\M(\cdots)$ is denoted by $\M^k(\cdots).$ 
\end{theorem}
\begin{proof}
See page~\pageref{proof:Moduli}.
\end{proof}
\begin{theorem}\label{ModuliS1}
In the situation of Theorem~\ref{ModuliS}, the set $\M^0 (\phi,\{H_s,G_s,J_s\},\gamma^\alpha,\gamma^\beta)$
is finite. Its cardinality modulo 2 is denoted by $ n(\phi,\{H_s,G_s,J_s\},\gamma^\alpha,\gamma^\beta).$
\end{theorem}
\begin{proof}
See below.
\end{proof}
\begin{theorem}\label{ModuliS2}
Let $\phi,H^\alpha,H^\beta,H_s,G_s, J^\alpha,J^\beta,J_s$ be as in Theorem~\ref{ModuliS}. Then the linear map
\[\Phi^{\beta\alpha}(H_s,G_s,J_s): CF_\ast \left(\phi,H^\alpha,J^\alpha,\hcl\right) \to CF_\ast \left(\phi,H^\beta,J^\beta,\hcl\right)\]
defined on generators by
\[\Phi^{\beta\alpha}(H_s,G_s,J_s)\left\langle \gamma^\alpha \right\rangle := \sum_{\begin{matrix}\gamma^ \beta\in \P^\hcl(\phi,H^\beta)\\ \mu(\gamma^\beta,\gamma^\alpha)=0  \end{matrix}} n(\phi,\{H_s,G_s,J_s\},\gamma^\alpha,\gamma^\beta) \left\langle \gamma^\beta\right\rangle\]
induces a homomorphism on homology (denoted by the same letter and called the \textbf{continuation map}). The homomorphism
\[\Phi^{\beta\alpha}(H_s,J_s): HF_\ast \Big(\phi,H^\alpha,J^\alpha,\hcl\Big) \to HF_\ast \left(\phi,H^\beta,J^\beta,\hcl\right)\]
is independent of the choice of continuation data. Moreover, it satisfies the composition formula
\[\Phi^{\delta\beta}\circ \Phi^{\beta\alpha}= \Phi^{\delta\alpha},\quad \Phi^{\alpha\alpha} =\OP{id}.\]
\end{theorem}
\begin{proof}
See below.
\end{proof}
\begin{corollary}\label{SlopeIso}
Let $\phi\in\OP{Symp}_c(W,\lambda/d)$ and $\hcl\in\pi_0\Omega_\phi.$ If $(H^\alpha,J^\alpha)$ and $(H^\beta,J^\beta)$ are regular Floer data for $\phi$, and $H^\alpha,H^\beta$ have the same slope, then
\[HF_\ast \left(\phi,H^\alpha,J^\alpha,\hcl\right) \cong HF_\ast \left(\phi,H^\beta,J^\beta,\hcl\right).\]
\end{corollary}
\begin{proof}[Proof of Theorems \ref{Moduli}, \ref{Moduli1}, \ref{Moduli2}, \ref{ModuliS}, \ref{ModuliS1} and \ref{ModuliS2}]\label{proof:Moduli} 
We prove that there exist compact subsets $K(\phi,H,J)$ and $K(\phi,\{H_s,G_s,J_s\})$ of $\hat{W}$ such that all elements of the Moduli spaces $\M(\phi,H,J,\gamma^-,\gamma^+)$ and $\M(\phi,\{H_s,G_s,J_s\},\gamma^\alpha,\gamma^\beta)$ map $\mathbb{R}^2$ into $K(\phi, H,J)$ and $K(\phi,\{H_s,G_s,J_s\}),$ respectively. Since $\M(\phi,H,J,\gamma^-,\gamma^+)$ is a special case of\\ $\M(\phi,\{H_s,G_s,J_s\},\gamma^\alpha,\gamma^\beta)$ (for $H_s,J_s$ $s$-independent, $G_s=0$, $\gamma^-=\gamma^\alpha,$ and $\gamma^+=\gamma^\beta$), we consider only $\M(\phi,\{H_s,G_s,J_s\},\gamma^\alpha,\gamma^\beta).$

Let $r_0\in(1,\infty)$ be such that $G_{s,t}(x,r)=0, H_{s,t}(x,r)=a(s,t)$ and conditions \eqref{JCondInfComp} and \eqref{JCondInfDel} hold for $J_s(x,r),$ for all $s,t\in\mathbb{R},$ and $(x,r)\in M\times(r_0,\infty).$ The existence of such $r_0$ is guaranteed by the definition of continuation data (Definition~\ref{def:ContData}). Since the slopes of $H^\alpha$ and $H^\beta$ are admissible, the sets $\P_{\phi, H^\alpha}$ and $\P_{\phi,H^\beta}$ are subsets of the compact set $\hat{W}\setminus M\times (r_0,\infty).$ Assume there exist $\gamma^\alpha\in \P_{\phi,H^\alpha}, \gamma^\beta\in\P_{\phi, H^\beta},$ and $u\in \M(\phi,\{H_s,G_s,J_s\},\gamma^\alpha,\gamma^\beta)$ such that
\[u(\mathbb{R}^2)\cap M\times(r_0+1,\infty)\not=\emptyset.\]
The symplectomorphism $\phi$ is equal to the identity on $M\times(r_0,\infty).$ Therefore
\[u(s,t)=u(s,t+1)\]
whenever $u(s,t)\in M\times (r_0,\infty).$ This together with
\begin{align*}
\lim_{s\to-\infty}u(s,t)&=\gamma^\alpha(t),\\
\lim_{s\to+\infty}u(s,t)&=\gamma^\beta(t)
\end{align*}
implies
\[u([-C,C]\times[0,1])\cap M\times (r_0,\infty)= u(\mathbb{R}^2)\cap M\times(r_0,\infty)\]
for $C\in(0,\infty)$ large enough. Since $[-C,C]\times [0,1]$ is a compact set, the function $r\circ u$ attains maximum on $u^{-1}(M\times(r_0,\infty)).$ Let $U\subset u^{-1}(M\times(r_0,\infty))$ be the connected component of a point at which the maximum is achieved. Connected components of a locally path-connected space are open. Hence $U$ is an open subset of $u^{-1}(M\times(r_0,\infty)).$ Since $u^{-1}(M\times(r_0,\infty))$ is open in $\mathbb{R}^2,$ $U$ is an open subset of $\mathbb{R}^2$ as well. 

Lemma~\ref{thm:maxprinclin} implies that $\left.r\circ u\right|_U$ is constant. Denote by $c\in(0,\infty)$ the value of $r\circ u$ on $U.$ Therefore $U$ is a connected component of $u^{-1}(M\times\{c\}).$ It follows that $U$ is a closed subset of $u^{-1}(M\times\{c\}).$ Since $u^{-1}(M\times\{c\})$ is closed in $\mathbb{R}^2,$ so is $U.$ This implies that $U$ is a non-empty open and closed subset of $\mathbb{R}^2,$ i.e. $U=\mathbb{R}^2.$ However, this is impossible because 
\[\lim_{s\to+\infty}u(s,t)=\gamma^\beta(t).\] 
Since everything happens in a compact subset of $\hat{W},$ the arguments for the case of a closed symplectic manifold (instead of $\hat{W}$) apply here as well. 

We end the proof by two remarks. Bubbling-off cannot occur because the symplectic form on $\hat{W}$ is exact. Lemma~\ref{EId} implies that elements of $\M(\phi,\{H_s,G_s,J_s\},\gamma^\alpha,\gamma^\beta)$ have uniformly bounded energy.
\end{proof}

\begin{remark}\label{def:HF}
Let $\phi\in\OP{Symp}_c(W,\lambda/d)$ and let $\hcl\in\Omega_\phi.$ For a fixed admissible slope, the groups $HF_\ast(\phi,H,J,\hcl)$ are canonically isomorphic to each other via the continuation maps. Hence we can write
\[HF_\ast(\phi,a,\hcl).\]
We define
\[HF_\ast (\phi,a):=\bigoplus_{\hcl\in \pi_0\Omega_\phi}HF_\ast(\phi,a,\hcl).\]
\end{remark}

\begin{definition}
Let $\phi$ and $\hcl$ be as in Remark~\ref{def:HF}. The continuation maps give rise to a homomorphism (which we call the same)
\begin{equation}\label{eq:contmap}
HF_\ast(\phi,a_1,\hcl)\to HF_\ast(\phi, a_2,\hcl)
\end{equation}
whenever $a_1\leqslant a_2.$ The groups $HF_\ast(\phi,a,\hcl),$ where $a\in\mathbb{R}$ goes through all admissible slopes, together with homomorphisms~\eqref{eq:contmap} form a directed system of groups. We define
\begin{align*}
& HF_\ast(\phi,\infty, \hcl):= \underset{a}{\lim_{\longrightarrow}} HF_\ast(\phi,a,\hcl),\\
& HF_\ast(\phi,\infty):= \bigoplus_{\hcl\in \pi_0\Omega_\phi} HF_\ast(\phi,\infty,\hcl).
\end{align*}
\end{definition}

\begin{lemma}\label{thm:VarSl}
Let $a_1$ and $a_2$ be real numbers such that all numbers in the interval $[a_1,a_2]$ are admissible. Then
\[HF_\ast(\phi,a_1,\hcl)\cong HF_\ast(\phi,a_2,\hcl)\]
for all $\phi\in\OP{Symp}_c(W,\lambda/d)$ and $\hcl\in\pi_0\Omega_\phi.$ Moreover, the isomorphism is realized by the continuation map. 
\end{lemma}
\begin{proof}
Let $(H,J)$ be regular Floer data for $(\phi,a_1).$ Assume that conditions \eqref{HamLin}, \eqref{JCondInfComp} and \eqref{JCondInfDel} hold on $M\times (r_0,\infty).$ Let $\chi: (r_0,\infty)\to [0,1]$ be a smooth function equal to 0 near $r_0$ and 1 near $\infty,$ and such that $\dot{\chi}\in [0,1].$ Consider the Hamiltonian $\widetilde{H}: \mathbb{R}\times\hat{H}\to\mathbb{R}$ defined by
\[\widetilde{H}_t(x,r):=a_1+(a_2-a_1)\chi(r),\]
for $(x,r)\in M\times(r_0,\infty),$ and $\widetilde{H}=H$ on $\hat{W}-M\times(r_0,\infty).$ The function $r\circ u$ satisfies the maximum principle, where $u$ is a solution of the Floer equation corresponding to the pair $(\widetilde{H},J)$ (see Proposition~4.1 in \cite{BegOverview-Wendl}). Moreover, there are no elements of $\P_{\phi,\widetilde{H}}$ intersecting $M\times (r_0,\infty).$ This implies that all ingredients for the definition of $HF_\ast(\phi,\widetilde{H},J,\theta)$ are contained in the compact set $\hat{W}-M\times (r_0,\infty).$ However, $H$ and $\widetilde{H}$ coincide there. Hence, the chain complexes $CF_\ast(\phi,H,J,\theta)$ and $CF_\ast(\phi,\widetilde{H},J,\theta)$ are identical. This finishes the proof.
\end{proof}

\begin{remark}\label{rmk:KnownSlopes}
In case $\phi=\OP{id},$ we can identify the groups $HF(\phi,a),$ for certain slopes, with known groups. By construction, $HF(\OP{id,\infty})$ is equal to the symplectic homology of $W$ with $\mathbb{Z}_2$ coefficients
\[HF(\OP{id},\infty)= SH(W;\mathbb{Z}_2).\]
Since the Floer homology for $C^2$-small Hamiltonian reduces to Morse homology, we get
\[HF_{\ast}(\OP{id},\varepsilon)\cong H_{\ast+n}(W,M;\mathbb{Z}_2)\quad \text{and}\quad HF_{\ast}(\OP{id},-\varepsilon)\cong H_{\ast+n}(W;\mathbb{Z}_2)\]
for $\varepsilon>0$ small enough. Also
\[HF_\ast(\OP{id},\varepsilon,\hcl)=0\]
for $\hcl\in\pi_0\Omega_\phi$ non-trivial and $\abs{\varepsilon}$ small.
\end{remark}

\subsection{Naturality}\label{sec:Naturality}
Let $\{\psi_t\}_{t\in\mathbb{R}}$ be an isotopy of $\hat{W},$ let $H:\mathbb{R}\times\hat{W}\to \mathbb{R}:(t,x)\mapsto H_t(x)$ be a smooth function, and let $\{J_t\}_{t\in\mathbb{R}}$ be a family of almost complex structures on $\hat{W}.$ We define $\psi^\ast H$ to be the smooth function
\[\psi^\ast H\: : \: \mathbb{R}\times\hat{W}\to \mathbb{R}\::\: (t,x)\mapsto H(t,\psi_t(x)),\]
and $\psi^\ast J$ to be the family $\{\psi^\ast_t J_t\}_{t\in\mathbb{R}}.$
Let $\phi:\hat{W}\to\hat{W}$ be a diffeomorphism. Assume that the isotopy $\psi$ satisfies
\[\psi_t\circ\phi\circ\psi_1=\phi\circ\psi_{t+1}.\]
Then the map
\[\Omega_\phi\to \Omega_{\phi\circ\psi_1}\::\:\gamma(\cdot)\mapsto \psi_{\cdot}^{-1}\circ\gamma(\cdot)=:\psi^\ast\gamma\]
is well defined. It induces the bijection
\[\pi_0\Omega_\phi\to\pi_0\Omega_{\phi\circ\psi_1}\::\:\hcl\mapsto \psi^\ast\hcl. \]
\begin{lemma}\label{thm:Nat}
Let $\phi\in\OP{Symp}_c(W,\lambda/d),$ let $K_t:\hat{W}\to\mathbb{R}$ be a Hamiltonian that is linear near infinity, satisfies condition \eqref{HamTwist}, and such that $\psi_1^K\in\OP{Symp}_c(W,\lambda/d)$, let $\hcl\in\pi_0\Omega_\phi,$ and let $(H,J)$ be regular Floer data for $\phi.$ Then the Hamiltonian $\left(\psi^K\right)^\ast(H-K)$ and the almost complex structure $\left(\psi^K\right)^\ast J$ are regular Floer data for $\phi\circ \psi^K_1.$ Moreover, the homomorphism
\[\N(K):CF(\phi,H,J,\hcl)\to CF\left(\phi\circ \psi^K_1, \left(\psi^K\right)^\ast(H-K),\left(\psi^K\right)^\ast J,\left(\psi^K\right)^\ast\hcl\right)\]
defined on generators by
\begin{equation}\label{NatEq:1}
\gamma \mapsto  \left(\psi^K \right)^\ast\gamma 
\end{equation}
is an isomorphism of chain complexes. Note that this isomorphism does not preserve the grading. However, the relative index is preserved.
\end{lemma}
\begin{proof}
Since $K$ is twisted by $\phi,$ the Hamiltonian  isotopy $\psi_t^K$ satisfies $ \psi^K_t\circ\phi\circ\psi^K_1=\phi\circ \psi_{t+1}^K.$ This further implies that $\left(\psi^K\right)^\ast (H-K)=:G$ and $\left(\psi^K\right)^\ast J=:\widetilde{J}$ are twisted by $\phi\circ\psi_1^K.$ Near infinity, $\psi_t^K$ is equal to $(x,r)\mapsto (\sigma^\beta_{at}(x),r)$ for some real number $a.$ Hence, $\psi_t^K$ preserves $\xi^\beta,R^\beta,$ and $r.$ Therefore, the conditions on $G$ and $\widetilde{J}$ near infinity are met. It is easy to check that $\gamma$ is an element of $\P_{\phi,H}$ if, and only if, $\left(\psi^K\right)^\ast \gamma$ is an element of $\P_{\phi\circ\psi^K_1,G}.$ Moreover, $u$ is a solution of the Floer equation~\eqref{FloerEq} if, and only if, $(s,t)\mapsto \left(\psi_t^K\right)^{-1} \circ u(s,t)$ is a solution of the Floer equation
\[\partial v +\widetilde{J}_t(v)\left(\partial_t v-X_{G_t}(v) \right)=0.\]
Now, we show that both complexes $CF(\phi,H,J,\hcl)$ and $CF\left(\phi\circ \psi^K_1, G,\widetilde{J},\left(\psi^K\right)^\ast\hcl\right)$ are graded by the same group, i.e. we show $N_\hcl= N_{\left(\psi^K\right)^\ast\hcl}.$ The map
\[\hat{W}_{\phi\circ\psi^K_1}\to \hat{W}_\phi\::\: [(x,t)]\mapsto [(\psi^K_t(x),t)]\]
induces a symplectic vector bundle isomorphism $T^V\hat{W}_{\phi\circ\psi^K_1}\to T^V\hat{W}_\phi.$ It also induces a bijection between $\pi_1(\Omega_{\phi\circ\psi_1^K},\left(\psi^K\right)^\ast\hcl)$ and $\pi_1(\Omega_\phi,\hcl).$ Hence, the groups 
\[\left\langle c_1(T^V\hat{W}_{\phi\circ\psi^K_t}), \pi_1(\Omega_{\phi\circ\psi_1^K},\left(\psi^K\right)^\ast\hcl)\right\rangle = \mathbb{Z}N_{\left(\psi^K\right)^\ast\hcl} \quad\text{and}\quad \left\langle c_1(T^V\hat{W}_{\phi}), \pi_1(\Omega_{\phi},\hcl)\right\rangle = \mathbb{Z}N_{\hcl}\] are identical.

Finally, let $\gamma_0,\gamma_1\in \P_{\phi,H}^\hcl$ and let $u:\mathbb{R}\times\mathbb{R}\to \hat{W}$ be a homotopy between them as in Definition~\ref{RMasl}. Then $\left(\psi^K\right)^\ast u:(s,t)\mapsto \left(\psi^K_t\right)^{-1}\circ u(s,t)$ is a homotopy between $\left(\psi^K\right)^\ast\gamma_0,\left(\psi^H\right)^\ast\gamma_1\in \P_{\phi,H}^{(\psi^K)^\ast\hcl},$ and
\[\mu\left(u\right)=\mu\left(\left(\psi^K\right)^\ast u\right).\]
Therefore $\mu(\gamma_0,\gamma_1)=\mu\left( \left(\psi^K\right)^\ast\gamma_0,\left(\psi^H\right)^\ast\gamma_1\right),$ and the proof is finished.
\end{proof}
\begin{remark}
The isomorphism $\N(K)$ from Lemma~\ref{thm:Nat} induces an isomorphism on the homology level. That isomorphism will be denoted the same, i.e. by 
\[\N(K):HF(\phi,H,J,\hcl)\to HF\left(\phi\circ \psi^K_1, \left(\psi^K\right)^\ast(H-K),\left(\psi^K\right)^\ast J,\left(\psi^K\right)^\ast\hcl\right).\]
\end{remark}

\begin{lemma}\label{thm:NPhiIntertwined}
Let $\phi\in\OP{Symp}_c(W,\lambda/d),$ let $(H^\alpha,J^\alpha),(H^\beta,J^\beta)$ be regular Floer data for $\phi$ with the slope of $H^\beta$ greater than or equal to the slope of $H^\alpha,$ let $\hcl\in\pi_0\Omega_\phi,$ and let $K$ be a Hamiltonian satisfying conditions \eqref{HamTwist} and \eqref{HamLin} such that $\psi_1^K\in\OP{Symp}_c(W,\lambda/d)$. Then the following diagram commutes
\begin{center}
\begin{tikzcd}
HF(\phi,H^\alpha, J^\alpha,\hcl)\arrow{r}{\Phi^{\beta\alpha}}\arrow{d}{\N(K)} & HF(\phi,H^\beta, J^\beta,\hcl)\arrow{d}{\N(K)}\\
HF\left( \phi\circ\psi_1^K,\cdots\right) \arrow{r}{\Phi^{\beta\alpha}}& HF \left(\phi\circ\psi_1^K,\cdots \right). 
\end{tikzcd}
\end{center}
\end{lemma}
\begin{proof}
The proof is analogous to the proof of Lemma~\ref{thm:Nat}. The key ingredient is the following fact. The correspondence $u\leftrightarrow (\psi^K)^\ast u$ is a bijection between the sets of solutions  of the appropriate $s$-dependent Floer equations. Here, $(\psi^K)^\ast u(s,t):= (\psi^K_t)^{-1}\circ u(s,t).$
\end{proof}
\begin{proposition}\label{prop:nat}
Let $\phi\in\OP{Symp}_c(W,\lambda/d),$ let $a\in(-\infty,\infty]$ and let $K_t:\hat{W}\to\mathbb{R}$ be a Hamiltonian such that $K_t\circ\phi=K_{t+1},$  $K_t(x,r)=rk,\:k\in\mathbb{R}$ near infinity, and such that $\psi_1^K\in\OP{Symp}_c(W,\lambda/d)$.Then $\N(K)$ induces an isomorphism (denoted the same)
\begin{equation}\N(K)\::\:HF(\phi,a)\to HF(\phi\circ\psi_1^K,a-k).
\end{equation}
\end{proposition}
\begin{proof}
The proposition follows from Lemma~\ref{thm:NPhiIntertwined}.
\end{proof}
Let $H_t$ and $K_t$ be Hamiltonians $\hat{W}\to\mathbb{R}$ whose Hamiltonian isotopies are well defined (for all times). We denote the Hamiltonians $-H_t\circ\psi^H_t$ and $H_t+K_t\circ\left(\psi_t^H\right)^{-1}$ by $\overline{H}$ and $H\#K,$ \label{notation:sharp} respectively.  Note that 
\[\psi_t^{\overline{H}}=\left(\psi_t^H\right)^{-1}\quad\text{and}\quad \psi_t^{H\#K}=\psi^H_t\circ\psi_t^K.\]
\begin{remark}\label{NatComp}
Let $\phi\in\OP{Symp}_c(W,\lambda/d),$ let $H_t$ and $K_t$ be Hamiltonians linear near infinity. Assume that $(\phi, H)$ and $\left(\phi\circ \psi^H_1,K\right)$ satisfy condition \eqref{HamTwist} and that $\psi^H_1,\psi^K_1\in \OP{Symp}_c(W,\lambda/d)$. Then $\left( \phi\circ\psi_1^H,\overline{H}\right)$ and $(\phi,H\# K)$ satisfy condition \eqref{HamTwist} as well, $\overline{H}$ and $ H\# K$ are linear near infinity, and their time-1 maps are in the group $\OP{Symp}_c(W,\lambda/d).$ Moreover, 
\begin{equation}
\N(K)\circ\N(H)=\N(H\# K),\quad\N(H)^{-1}=\N(\overline{H}).
\end{equation}
\end{remark}

\subsection{Invariance under isotopies}\label{sec:InvIso}
The goal of this section is to associate an isomorphism 
\[I(\{\phi_t\}):HF(\phi_0,a)\to HF(\phi_1,a)\]
to a compactly supported isotopy (through $\OP{Symp}_c(W,\lambda/d)$) $\{\phi_t\}_{t\in[0,1]}$ between $\phi_0$ and $\phi_1\in \OP{Symp}_c(W,\lambda/d).$ Without loss of generality, we may assume $\phi_t=\phi_0$ for $t$ near $0$ and $\phi_t=\phi_1$ for $t$ near 1. By Lemma~\ref{thm:PathExSymp} below, $\phi_t=\psi^K_t\circ\phi_0$ for some compactly supported Hamiltonian $K:[0,1]\times\hat{W}\to \mathbb{R}$ (note that $K_t=0$ for $t$ near 0 or 1).

\begin{definition}
Let $\{\phi_t\}$ and $K$ be as above. Consider the Hamiltonian $K^0:\mathbb{R}\times\hat{W}\to \mathbb{R}$ that satisfies $K^0_t\circ\phi_0=K^0_{t+1}$ and
\[K^0_t=K_t\circ\phi_0\]
for $t\in[0,1].$ This Hamiltonian generates the isotopy $\phi_0^{-1}\circ\phi_t.$ We define
\[I(\{\phi_t\}):=\N(K^0)\::\:HF(\phi_0,a)\to HF(\phi_1,a).\]
\end{definition}

\begin{lemma}\label{thm:PathExSymp}
Any path $\phi_t\in \OP{Symp}_c(W,\lambda/d)$ is a Hamiltonian isotopy generated by a Hamiltonian compactly supported in $W-M.$
\end{lemma}
\begin{proof}
Let $X_t$ be a vector field of $\phi_t.$ For simplicity, we denote $F_{\phi_t}$ by $F_t.$ By applying Cartan's formula on
\[\phi_t^\ast\lambda-\lambda = d F_t,\]
we get
\[\phi^\ast_t(X_t\lrcorner d\lambda+d(\lambda(X_t)))=d(\partial_tF_t).\]
Hence,
\[X_t\lrcorner \omega=d\left( \partial_tF_t\circ\phi_t^{-1}-\lambda(X_t)\right),\]
and the proof is finished.
\end{proof}
\begin{theorem}\label{IsoInv}
Let $\phi_0,\phi_1$ be exact symplectomorphisms $\hat{W}\to\hat{W}$ in the same connected component of $\OP{Symp}_c(W,\lambda/d).$ Then, for each compactly supported isotopy through $\OP{Symp}_c(W,\lambda/d)$ between $\phi_0$ and $\phi_1$ and for every $-\infty<a\leqslant \infty,$ there exists an isomorphism
\begin{equation}\label{IsoInvEq:1} HF(\phi_0, a)\to HF(\phi_1, a). 
\end{equation}
The isomorphism depends only on the homotopy class of the isotopies between $\phi_0$ and $\phi_1.$ 
\end{theorem}
\begin{proof}
The existence of the isomorphism follows from the discussion above (the proof is the same as in~\cite{MR1297130} (Theorem~3.3)). Lemma~\ref{thm:ContHamLoop} below proves that the isomorphism depends only on the homotopy class of the isotopies between $\phi_0$ and $\phi_1.$ 
\end{proof}
\begin{lemma}\label{thm:ContHamLoop}
Let $\phi\in \OP{Symp}_c(W,\lambda/d),$ let $\hcl\in\pi_0\Omega_\phi,$ let $(H,J)$ be admissible Floer data for $\phi,$ and let $K_t:\hat{W}\to\mathbb{R}$ be a compactly supported Hamiltonian such that $K_t\circ\phi=K_{t+1}$ and such that it generates a loop of Hamiltonian diffeomorphisms that is contractible in the group of all Hamiltonian diffeomorphisms compactly supported in $W-M.$ Then the naturality isomorphism
\[\N(K)\::\:HF(\phi,H,J,\hcl)\to HF\left(\phi,(\psi^K)^\ast(H-K),(\psi^K)^\ast J,(\psi^K)^\ast\hcl\right)\]
coincides with the continuation morphism (which is well defined since $H$ and $(\psi^K)^\ast(H-K)$ have the same slope).  
\end{lemma}
\begin{proof}
By the assumptions, there exists a family $\psi_{s,t}:\hat{W}\to\hat{W},\:s\in\mathbb{R},\: t\in [0,1],$ of Hamiltonian diffeomorphisms compactly supported in $W-M$ such that $\psi_{s,0}=\psi_{s,1}=\OP{id},$ for all $s\in\mathbb{R},$ and such that
\[\psi_{s,t}=\left\lbrace \begin{matrix} \psi^K_{t}&\text{for } s>>0,\\ \OP{id} &\text{for } s<<0 \end{matrix}\right..\]  
We extend $\psi_{s,t}$ to an $\mathbb{R}\times\mathbb{R}$-family by requiring that it satisfies the following periodicity condition
\[\psi_{s,t}\circ\phi=\phi\circ\psi_{s,t+1}.\]
Note that $\psi_{s,t}=\psi^K_{t}$ holds for $s>>0$ after the extension as well.
Let $\gamma^\alpha,\gamma^\beta\in\P^\hcl(\phi, H)=\P(\phi, H)\cap \hcl.$ The map $u\mapsto \psi^\ast u,$ where $\psi^\ast u:(s,t)\mapsto \psi^{-1}_{s,t}\circ u(s,t),$ defines a bijection between the sets 
\[\M^0(\phi,\{H,0,J\}, \gamma^\alpha,\gamma^\beta) \quad \text{and}\quad \M^0(\phi,\{(\psi_{s,\cdot})^\ast(H-K_s), -(\psi_{s,\cdot})^\ast K_s,(\psi_{s,\cdot})^\ast J\}, \gamma^\alpha, (\psi^K)^\ast\gamma^\beta)\]
(for the definition of $\psi_{s,\cdot}^\ast(\cdots)$ see the beginning of section~\ref{sec:Naturality}). The former is empty unless $\gamma^\alpha=\gamma^\beta,$ and in that case, it consists of a single element, namely $(s,t)\mapsto \gamma^\alpha(t)$ (the reason for this is that $(H,0,J)$ is independent of $s$). Hence
\[n(\phi,\{(\psi_{s,\cdot})^\ast(H-K_s), -(\psi_{s,\cdot})^\ast K_s,(\psi_{s,\cdot})^\ast J\}, \gamma^\alpha, (\psi^K)^\ast\gamma^\beta)=\left\lbrace \begin{matrix}
1&\text{if } \gamma^\alpha=\gamma^\beta,\\ 0&\text{otherwise.}
\end{matrix}\right.\]
In other words, the continuation map associated to the continuation data
\[\{(\psi_{s,\cdot})^\ast(H-K_s), -(\psi_{s,\cdot})^\ast K_s,(\psi_{s,\cdot})^\ast J\}\]
sends the generator $\gamma$ of $CF(\phi,H,J,\hcl)$ to the generator $\left(\psi^K\right)^\ast\gamma$ of \[CF\left(\phi,(\psi^K)^\ast(H-K),(\psi^K)^\ast J,(\psi^K)^\ast\hcl\right),\] i.e. it coincides with the naturality isomorphism. This proves the lemma.
\end{proof}

\subsection{Action filtration}

In this section we define filtered Floer homology groups $HF_\ast^{<c}(\phi,a)$ for $\phi\in\OP{Symp}_c(W,\lambda/d),$ admissible slope $a\in\mathbb{R}$ and a real number $c\in\mathbb{R}.$ They will not be used in the rest of the paper.

Let $\phi,H,J,\hcl$ be as in Theorem~\ref{Moduli2}. The Action functional decreases along solutions of the Floer equation. Hence, the graded group
\[CF_\ast^{<c}(\phi,H,J,\theta)\]
generated by the elements of 
\[\left\{ \gamma\in\P^\hcl(\phi, H)\:|\: \A_{\phi,H}(\gamma)<c \right\}\]
is a subcomplex of $CF_\ast(\phi,H,J,\theta).$ Its homology is denoted by
\[HF_\ast^{<c}(\phi,H,J,\hcl).\]

Let $\phi,\hcl, H^\alpha,H^\beta,J^\alpha, J^\beta$ be as in Theorem~\ref{ModuliS}, and let $(H_s,0,J_s)$ be continuation data from $(H^\alpha,J^\alpha)$ to $(H^\beta,J^\beta).$ If $H_{s,t}(x)$ is increasing with respect to $s$ for all $t\in\mathbb{R}$ and $x\in\hat{W},$ then Lemma~\ref{EId} implies 
\[\A_{\phi, H^\alpha}(\gamma^\alpha)\leqslant \A_{\phi, H^\beta}(\gamma^\beta),\]
where $\gamma^\alpha\in \P(\phi, H^\alpha)$ and $\gamma^\beta\in \P(\phi, H^\beta)$ are such that  $\M(\phi,\{H_s,0,J_s\},\gamma^\alpha,\gamma^\beta)$ is nonempty. Therefore the continuation map
\[\Phi^{\beta\alpha}\::\:HF_\ast^{<c}(\phi,H^\alpha,J^\alpha,\hcl)\to HF_\ast^{<c}(\phi,H^\beta,J^\beta,\hcl)\]
is well defined whenever
\[H^\alpha_t(x)\leqslant H^\beta_t(x)\]
for all $t\in\mathbb{R}$ and $x\in\hat{W}.$

Given an admissible slope $a\in\mathbb{R}.$ Consider the subset $D(a)$ of all regular Floer data $(H,J)$ with the slope of $H$ equal to $a$ and such that $H_t(x)<0$ for $t\in\mathbb{R}$ and $x\in W.$ The set $D(a)$ is a directed set with preorder relation given by
\[(H,J)\leqslant (H',J')\quad \Longleftrightarrow \quad H\leqslant H'. \]

The groups $HF_\ast(\phi,H,J,\hcl)$ together with the homomorphisms $\Phi^{\beta\alpha}$ form a directed system of groups indexed by $D(a).$ We define
\begin{align*}
& HF^{<c}_\ast(\phi,a,\hcl):=\underset{(H,J)\in D(a)}{\lim_{\longrightarrow}}HF_\ast^{<c}(\phi,H,J,\hcl),\\
& HF_\ast^{<c}(\phi,a):=\bigoplus_{\hcl\in\pi_0\Omega_\phi} HF_\ast^{<c}(\phi,a,\hcl).
\end{align*}
Due to the following lemma, the naturality isomorphism is well defined for the groups $HF^{<c}(\phi,a).$

\begin{lemma}\label{lem:NatAct}
Let $H$ satisfy \eqref{HamTwist} and be equal $ar$ near infinity, and let $(\phi,G)$ be admissible. If $\gamma\in\P(\phi,G),$ then
\begin{equation}\label{NatActEq:1}
\A_{\phi\circ\psi^H_1,\left(\psi^H\right)^\ast(G-H)} \left( \left(\psi^H\right)^\ast\gamma \right)=\A_{\phi,G}(\gamma).
\end{equation}
(See section~\ref{sec:Naturality} for definition of $\left(\psi^H\right)^\ast(G-H)$ and $\left(\psi^H\right)^\ast\gamma.$)
\end{lemma}
\begin{proof}
Let $\widetilde{\gamma}(t):= \left(\psi^H\right)^\ast \gamma(t) = (\psi_t^H)^{-1}\circ \gamma(t).$
\begin{enumerate}[Step 1.]
\item The statement is true for $G=H.$\\
In this case, $\widetilde{\gamma}$ is a constant map $\widetilde{\gamma}\equiv x,$ and $\gamma(t)=\psi_t^H(x)$. Hence (because $\widetilde{\gamma}^\ast \lambda=0$), the left hand side of~\eqref{NatActEq:1} is equal to
\begin{equation}
-F_{\phi\circ\psi^H_1}(x)= -F_{\phi}\circ\psi_1^H(x)-F_{\psi^H_1}(x) = \A_{\phi,H}(\gamma). 
\end{equation}
\item Proof of the lemma.\\
We apply step~1 to pairs $(\phi,G)$ and $\left( \phi\circ\psi^H_1,\left(\psi^H\right)^\ast(G-H)\right),$ and use Remark~\ref{NatComp}.
\end{enumerate}
\end{proof}

The group $HF^{<c}(\phi,a)$ is invariant under isotopies of $\phi$ through $\OP{Symp}_c(W,\lambda/d).$ This follows from Lemma~\ref{lem:NatAct}. The isomorphism associated to an isotopy of $\phi$ through $\OP{Symp}_c(W,\lambda/d)$ can be constructed as in Section~\ref{sec:InvIso}.

\section{Applications to symplectic isotopy problem}
\subsection{Proofs of main results}
\begin{definition}\label{FibDTwist}[\cite{MR1765826,BiranGiroux}]
Assume the Reeb flow of $(M,\beta)$ is 1-periodic. Let $v:\mathbb{R}\to \mathbb{R}$ be a smooth function which is equal to $0$ on $(-\infty,0)$ and $-1$ on $(0.95,+\infty).$ The \textbf{right-handed fibered Dehn twist} is an element $\T$ of $\OP{Symp}_c(W,\omega)$ defined by
\[\T(p)=\left\lbrace \begin{matrix}
\left( \sigma^\beta_{v(r)}(x),r\right)&\text{for }p=(x,r)\in M\times (0,\infty),\\
p&\text{otherwise}.
\end{matrix}\right.\] 
The isotopy class of $\T$ does not depend on the choice of $v.$ 
\end{definition}
\begin{remark}\label{HamGenFib}
The fibered Dehn twist from Definition~\ref{FibDTwist}, is a time-1-map of the Hamiltonian
\[K:\widehat{W}\to\mathbb{R}\::\: p\mapsto \left\lbrace \begin{matrix}
-V(r)&\text{for }p=(x,r)\in M\times (0,\infty),\\
0&\text{otherwise}.
\end{matrix}\right.\]
Here, $V:\mathbb{R}\to\mathbb{R}$ is the unique primitive of $v$ equal 0 near 0. Note that \[K(x,r)= r-\int_0^1v(t)dt-1\]
near infinity.
\end{remark}
\begin{lemma}[Giroux]\label{thm:HomEquiv}
The inclusion
\[j\::\:\OP{Symp}_c(W,\lambda/d)\to \OP{Symp}_c(W,\omega)\]
is a homotopy equivalence.
\end{lemma}
\begin{proof}[Proof (following \cite{BiranGiroux})]
We construct a homotopy inverse of $j$ as follows. Let $Y^\phi,\:\phi\in\OP{Symp}_c(W,\omega)$ be a vector field on $W$ defined by
\[\lambda-\phi^\ast\lambda= \omega\left( Y^\phi,\cdot\right).\]
Since $Y^\phi$ is compactly supported in $W-M,$ its flow $\psi_t^\phi$ is well defined. We denote by $T$ the map
\[\OP{Symp}_c(W,\omega)\to\OP{Symp}_c(W,\lambda/d)\::\: \phi\mapsto \phi\circ\psi_1^\phi.\]

First, we prove that $T$ is well defined, i.e. that $T(\phi)\in \OP{Symp}_c(W,\lambda/d).$ Note that $\psi^\phi_t=:\psi_t$ is a symplectomorphism for all $t\in \mathbb{R}$ and that it preserves the form $\lambda-\phi^\ast \lambda.$ Using Cartan's formula, we get
\[\begin{split}
\frac{d}{dt} \left( \psi_t^\ast \lambda\right) &= \psi_t^\ast\left( Y^\phi\lrcorner d\lambda+ d Y^\phi\lrcorner \lambda\right)\\
&=\psi_t^\ast\left( \lambda-\phi^\ast\lambda+
d\left(\lambda\left( Y^\phi\right)\right) \right) \\
&=\lambda-\phi^\ast\lambda +d\left(\lambda\left( Y^\phi\right)\circ\psi_t \right).
\end{split}\] 
Hence
\begin{equation} \label{eq:les:Tmap1}
\psi_t^\ast\lambda-\lambda =\left( \lambda-\phi^\ast \lambda \right)t+ d\left( \int_0^t \lambda\left(Y^\phi \right)\circ \psi_s ds \right).
\end{equation}
This implies that
\begin{equation}
\label{eq:les:Tmap2}
\begin{split}
(\phi\circ\psi_1)^\ast \lambda-\lambda &= \psi_1^\ast \phi^\ast \lambda -\psi_1^\ast\lambda + \psi_1^\ast \lambda -\lambda\\
&= \psi_1^\ast\left(\phi^\ast\lambda -\lambda \right) + \psi_1^\ast\lambda -\lambda\\
&= \psi_1^\ast\lambda -\lambda -\left(\lambda-\phi^\ast \lambda\right)
\end{split}
\end{equation}
is an exact form.\\

Now, we prove that $j\circ T$ and $T\circ j$ are homotopic to the identity. The homotopies are given by
\[[0,1]\times \OP{Symp}_c(W,\omega)\to \OP{Symp}_c(W,\omega)\::\: (t,\phi)\mapsto \phi\circ \psi^\phi_t,\]
\[[0,1]\times\OP{Symp}_c(W,\lambda/d) \to \OP{Symp}_c(W,\lambda/d)\::\:(t,\phi)\mapsto \phi\circ\psi_t^\phi.\]
In what follows, we check that the latter is well defined. Similarly as in \eqref{eq:les:Tmap2}, we get
\[\left(\phi\circ \psi_t^\phi\right)^\ast\lambda- \lambda = \left( \left(\psi^\phi_t\right)^\ast\lambda -\lambda \right)+ \left(\phi^\ast\lambda -\lambda\right).\]
The both forms in the brackets (on the right-hand side of the equation above) are exact: the latter because of $\phi\in\OP{Symp}_c(W,\lambda/d),$ and the former because of \eqref{eq:les:Tmap1} in addition.
\end{proof}
\begin{proof}[Proof of Theorem~\ref{thm:SlopePeriodicity}]
Remark~\ref{HamGenFib} implies that the $\ell$-th power of the fibered Dehn twist $\tau^\ell$ is the time-1 map of a Hamiltonian $K:\hat{W}\to \mathbb{R}$ that is equal to $\ell r$ near infinity. The naturality provides us with the isomorphism (see Proposition~\ref{prop:nat})
\begin{equation}\label{eq:DTInfOrd1}
\N(K): HF(\OP{id},a+\ell)\overset{\cong}{\to} HF(\tau^\ell,a)
\end{equation}
for all admissible $a\in\mathbb{R}.$ On the other hand, if we assume that $\tau^\ell$ represents trivial class in $\pi_0\OP{Symp}_c(W,\omega),$ by Lemma~\ref{thm:HomEquiv}, it represents trivial class in $\OP{Symp}_c(W,\lambda/d)$ as well. Theorem~\ref{IsoInv} implies
\begin{equation}\label{eq:DTInfOrd2}
HF(\OP{id},a)\cong HF(\tau^\ell,a).
\end{equation}
From \eqref{eq:DTInfOrd1} and \eqref{eq:DTInfOrd2}, we get
\begin{equation}\label{eq:DTInfOrd3}
HF(\OP{id},a)\cong HF(\OP{id},a+\ell).
\end{equation}
(In general, the isomorphism~\eqref{eq:DTInfOrd3} preserves neither the grading nor the class in $\pi_0\Omega_\phi.$)
\end{proof}
\begin{remark}
Strictly speaking, the construction of the groups $HF_\ast(\phi,a)$ for $\phi\not=\OP{id}$ was not necessary for the proof of Theorem~\ref{thm:SlopePeriodicity}. Nevertheless, we find developing the theory for general exact symplectomorphism worthwhile for the following reasons. Using $HF_\ast(\phi,a)$ makes the proof transparent by placing the emphasis on an invariant of a mapping class. Additionally, the groups $HF_\ast(\phi,a)$ might be useful for further applications.

In the rest of the remark, we rewrite the proof of Theorem~\ref{thm:SlopePeriodicity}  without mentioning $HF_\ast(\phi,a)$ for $\phi\not=\OP{id}.$  
Let $K:\hat{W}\to\mathbb{R}$ be a Hamiltonian equal to $\ell r$ near infinity such that $\psi^K_1=\tau^\ell,$ and let $G_t:\hat{W}\to\mathbb{R}$ be a compactly supported Hamiltonian such that $G_{t+1}=G_t$ and $\psi^G_1=\tau^\ell.$ The Hamiltonian $G$ exists because $[\tau^\ell]=1\in\pi_0\OP{Symp}_c(W,\lambda/d)$ and because of Lemma~\ref{thm:PathExSymp}. Consider the Hamiltonian $G\#\overline{K}.$ It has the slope equal to $-\ell$ and generates the isotopy $\psi^G_t\circ(\psi^K_t)^{-1}.$ In particular, its time-1 map is equal to the identity. Now, the naturality with respect to this Hamiltonian provides the isomorphism~\eqref{eq:DTInfOrd3}.
\end{remark}
\begin{proof}[Proof of Corollary~\ref{thm:DTInfOrd}]
Assume, by contradiction, that the fibered Dehn twist represents a class of order $\ell\in\mathbb{N}$ in $\pi_0\OP{Symp}_c(W,\omega)$ (and, equivalently, in $\pi_0\OP{Symp}_c(W,\lambda/d)$). Then Theorem~\ref{thm:SlopePeriodicity} implies
\[HF(\OP{id},a)\cong HF(\OP{id},a+\ell)\]
for all admissible $a\in\mathbb{R}.$ In particular,
\[HF(\OP{id},k\ell-\varepsilon)\cong HF(\OP{id},-\varepsilon)\cong H(W;\mathbb{Z}_2),\]
for $k\in\mathbb{Z}$ and $\varepsilon>0$ small enough (see also Remark~\ref{rmk:KnownSlopes}). It follows that
\begin{equation}\label{eq:dim}
\dim HF(\OP{id},k\ell-\varepsilon)=\dim H(W;\mathbb{Z}_2).
\end{equation}
Since 
\[SH(W;\mathbb{Z}_2)=HF(\OP{id},\infty)=\underset{a}{\lim_{\longrightarrow}} HF(\OP{id},a)\]
has dimension greater than or equal to
\[d:=\dim H(W,\mathbb{Z}_2)+1,\]
there exist elements $\alpha_j\in HF(\OP{id},a_j),\:j\in\{1,\ldots, d\}$ such that 
\[\iota(\alpha_1),\ldots, \iota(\alpha_d)\in SH(W;\mathbb{Z}_2)\]
are linearly independent. Here, $\iota$ stands for the homomorphisms 
\[\iota :HF(\OP{id},a)\to \underset{a}{\lim_{\longrightarrow}} HF(\OP{id},a). \]
Let $\varepsilon>0$ be small enough, and let $k\in\mathbb{N}$ be such that
\[k\ell-\varepsilon>\max\{a_1,\ldots, a_d\}.\]
Then the images of $\alpha_1,\ldots, \alpha_d$ under the continuation maps
\[HF(\OP{id},a_j)\to HF(\OP{id},k\ell-\varepsilon),\quad j\in\{1,\ldots, d\}\]
are linearly independent. However, this leads to contradiction, because of \eqref{eq:dim} and $d>\dim H(W;\mathbb{Z}_2).$
\end{proof}

\begin{definition}
Let $m\in\mathbb{N}$ be a natural number. A $\mathbb{Z}_m$-graded vector space $V_\ast$ is called \textbf{symmetric} if there exists $k\in\mathbb{Z}_m$ such that 
\[V_{k-j}\cong V_j\]
for all $j\in\mathbb{Z}_m.$
\end{definition}

\begin{corollary}\label{thm:DTNonTriv}
Let $W,M,\beta$ be as in section~\ref{sec:Conv}, let $o\in\pi_0\Omega_{\OP{id}}$ be the class of contractible loops, and let $N$ be the minimal Chern number $N_o$ (see page~\pageref{eq:ChernNum}). Assume that the Reeb flow of $(M,\beta)$ induces a free circle action on $M.$  If the homology $H_\ast(W;\mathbb{Z}_2)$ rolled up modulo $2N$ is not symmetric, then the fibered Dehn twist represents a nontrivial class in $\pi_0\OP{Symp}_c(W,\omega).$  
\end{corollary}
\begin{proof}
Let $1>\varepsilon>0.$ Theorem~\ref{thm:SlopePeriodicity} implies
\[HF(\OP{id},-\varepsilon)\cong HF(\OP{id},1-\varepsilon).\]
This isomorphism preserves neither the grading nor the class in $\pi_0\Omega_{\OP{id}}.$ However, it preserves the relative grading. Hence there exist a class $o'\in \pi_0\Omega_{\OP{id}}$ and $c\in\mathbb{Z}_{2N}$ such that
\begin{equation}\label{eq:oo}
HF_\ast(\OP{id},-\varepsilon,o)\cong HF_{\ast+c}(\OP{id}, 1-\varepsilon,o').
\end{equation}
Since $\varepsilon\in(-1,0),\: 1-\varepsilon\in(0,1)$ and numbers in $(-1,0)\cup(0,1)$ are all admissible, we get (by Lemma~\ref{thm:VarSl} and Remark~\ref{rmk:KnownSlopes})
\begin{align*}
HF_\ast(\OP{id},-\varepsilon,o)&\cong H_{\ast+n}(W;\mathbb{Z}_2),\\
HF_\ast(\OP{id},1-\varepsilon,o')&\cong\left\{ \begin{matrix} H_{\ast+n}(W,M;\mathbb{Z}_2)&\text{if }o=o' \\ 0&\text{otherwise.}\end{matrix}\right.
\end{align*} 
The group $H(W;\mathbb{Z}_2)$ is never 0. Hence (because of \eqref{eq:oo}), $o=o'$ and 
\[H_\ast(W;\mathbb{Z}_2)\cong H_{\ast+c}(W,M;\mathbb{Z}_2).\]
By Poincar\'e duality
\[H_\ast(W,M;\mathbb{Z}_2)\cong H^{2n-\ast}(W;\mathbb{Z}_2).\]
Since we are working with field coefficients, we have
\[H^{2n-\ast}(W,\mathbb{Z}_2)\cong \OP{Hom}(H_{2n-\ast}(W;\mathbb{Z}_2),\mathbb{Z}_2)\cong H_{2n-\ast}(W;\mathbb{Z}_2).\] 
Therefore
\[H_{2n-c-\ast}(W;\mathbb{Z}_2)\cong H_{\ast}(W;\mathbb{Z}_2).\]
We take $k=2n-c$ and the proof is finished.
\end{proof}

\begin{proof}[Proof of Corollary~\ref{cor:DTNonTrivC0}]
The corollary is a special case of Corollary~\ref{thm:DTNonTriv} (the case in which $N=\infty$).
\end{proof}

\subsection{Examples}
\begin{example}\label{SymSpaces}
Let $(Q,g)$ be a closed Riemannian manifold. We denote by $D^\ast Q$ and $S^\ast Q$ the cotangent disk bundle and the unit cotangent bundle of $Q,$ respectively. The standard Liouville form $\lambda_{\OP{can}}$ on $T^\ast Q$ equips $D^\ast Q$ with the structure of a Liouville domain. The Reeb flow on $S^\ast Q$ coincides with the geodesic flow of $Q$ under the obvious identification of tangent and cotangent bundles of $Q.$ It is periodic if, and only if, the geodesics of $(Q,g)$ are all periodic. Examples of Riemannian manifolds with all geodesics periodic are spheres $\mathbb{S}^m,$ complex and quaternionic projective spaces $\mathbb{C}P^m$ and $\mathbb{H}P^m,$ and the Cayley plane (octonionic projective plane) $\mathbb{C}aP^2$ with the standard metrics.

Assume $(Q,g)$ is one of these manifolds. We can rescale $g$ so that the Reeb flow on $S^\ast Q$ is 1-periodic. By a theorem of Viterbo (see Theorem~\ref{thm:Vit} below), the symplectic homology $SH_\ast (D^\ast Q;\mathbb{Z}_2)$ is isomorphic to the singular homology $H_\ast (\Lambda Q;\mathbb{Z}_2)$ of the free loop space of $Q.$ The homology of $\Lambda Q$ is explicitely computed in~\cite{MR0649625}, and it turns out to be infinite dimensional in $\mathbb{Z}_2$ coefficients. Hence Corollary~\ref{thm:DTInfOrd} implies that the corresponding fibered Dehn twist is of infinite order in $\pi_0\OP{Symp}_c(D^\ast Q,d \lambda_{\OP{can}}).$   
\end{example}
\begin{theorem}[Viterbo]\label{thm:Vit}
Let $(Q,g)$ be a closed Riemannian manifold, and let $\Lambda Q$ be its free loop space. If $(W,\lambda)=\left(D^\ast Q,\lambda_{\OP{can}} \right),$ then
\[SH_\ast(W;\mathbb{Z}_2)\cong H_\ast(\Lambda Q;\mathbb{Z}_2).\] 
\end{theorem}
\begin{proof} 
Proofs can be found in \cite{Viterbo}, \cite{MR2190223}, \cite{MR2276534}, and \cite{Abouzaid}.
\end{proof}
\begin{remark}
Corollary~\ref{thm:DTNonTriv} does not imply nontriviality of the fibered Dehn twists for the Liouville domains considered in Example~\ref{SymSpaces}. The reason is the following. Let $(Q,g)$ be as in Example~\ref{SymSpaces}. The disk cotangent bundle $D^\ast Q$ is homotopic to $Q$ which is a closed manifold, and therefore (by Poincar\'e duality), the homology $H_\ast(D^\ast Q; \mathbb{Z}_2)$ is symmetric.
\end{remark}
\begin{remark}
The square of the Dehn-Seidel twist \cite{MR1362824, thesis-Seidel} for $\mathbb{S}^m$ is symplectically isotopic to the fibered Dehn twist on $D^\ast \mathbb{S}^m,\:m\in\mathbb{N}.$
\end{remark}
\begin{remark}
The fibered Dehn twists on $D^\ast\mathbb{S}^2$ and $D^\ast\mathbb{S}^6$ are smoothly isotopic to the identity relative to the boundary (see Lemma~6.3 in \cite{MR1743463} and also Lemma~10.2 in \cite{Avdek}). The same is true for $D^\ast \mathbb{C}P^m,\,m\in\mathbb{N}$ \cite[Proposition~4.6]{MR1765826}. On the other hand, the fibered Dehn twist on $D^\ast\mathbb{S}^m$ for $m\not\in\{2,6\}$ is not even smoothly isotopic to the identity relative to the boundary (see, for example, Lemma~10.1 in \cite{Avdek}). Interestingly, the only manifolds among $\mathbb{S}^m$ and $\mathbb{C}P^m,\:m\in\mathbb{N},$ admitting almost complex structures are $\mathbb{S}^2,\mathbb{S}^6$ and $\mathbb{C}P^m,\:m\in\mathbb{N}.$ This is not a coincidence (see Proposition~\ref{prop:SmoothlyTrivial} below).  
\end{remark}

\begin{proposition}\label{prop:SmoothlyTrivial}
Let $(Q,g)$ be a Riemannian manifold whose all geodesics are 1-periodic. Assume there exists an almost complex structure $J$ on $Q$ that preserves the norm of vectors in $TQ.$ Then the fibered Dehn twist on $D^\ast Q$ is smoothly isotopic to the identity relative to the boundary.
\end{proposition}
It is an interesting question whether in the situation of Proposition~\ref{prop:SmoothlyTrivial} the fibered Dehn twist itself is smoothly isotopic to the identity relative to the boundary.

\begin{lemma}\label{lem:GeoFlowCont}
Let $(Q,g)$ be as in Proposition~\ref{prop:SmoothlyTrivial}. Denote by $SQ\subset TQ$ the unit tangent bundle. Then the geodesic flow induces a loop of diffeomorphisms $SQ\to SQ$ whose square is contractible.
\end{lemma}
\begin{proof}
Let $J$ be as in Proposition~\ref{prop:SmoothlyTrivial}, and let $\gamma_\xi$ be the unique geodesic of $Q$ such that $\dot{\gamma}_\xi (0)= \xi\in TQ.$ The (restriction of the) geodesic flow is given by
\[\Psi_t\::\: SQ\to SQ\::\: \xi\mapsto \dot{\gamma}_\xi(t).\]
The antipodal map 
\[SQ\to SQ\::\:\xi\mapsto -\xi\]
is isotopic to the identity via isotopy
\[[0,\pi]\times SQ\to SQ\::\:(s,\xi)\mapsto \cos{s}\xi+\sin{s}J\xi.\]
Hence the loop
\[\mathbb{R}/\mathbb{Z}\:\ni\:t\mapsto \Psi_t\:\in\:\OP{Diff}(SQ)\]
is homotopic to the loop
\[t\mapsto -\Psi_t(-\cdot)\]
($\Psi_t$ is pre- and postcomposed by the antipodal map). Since $-\Psi_t(-\xi)=\Psi^{-1}_t(\xi),\:\xi\in SQ,$ the loop $t\mapsto \Psi_t$ is homotopic to its inverse $t\mapsto \Psi_t^{-1}.$ Therefore, the loop
$t\mapsto \Psi^2_t$
is contractible. 
\end{proof}

\begin{proof}[Proof of Proposition~\ref{prop:SmoothlyTrivial}]
We take $(W,\lambda)=(D^\ast Q,\lambda_{\OP{can}})$ and use the notation from section~\ref{sec:Conv}. The square of the fibered Dehn twist is given by
\begin{equation}\label{eq:SqDTw}
\hat{W}\:\ni\:p\mapsto  \left\{ \begin{matrix}
\left( \sigma^2_{v(r)}(x),r\right)&\text{for }p=(x,r)\in M\times (0,\infty),\\
p&\text{otherwise},
\end{matrix}\right.
\end{equation}
where $\sigma_t=\sigma^\beta_t$ is the Reeb flow on $M, $ and $v:\mathbb{R}\to\mathbb{R}$ is a smooth function that is equal to 0 on $(-\infty,0)$ and -1 on $(0.95,+\infty).$  By Lemma~\ref{lem:GeoFlowCont}, $\sigma^2$ seen as a loop
\[\mathbb{R}/\mathbb{Z}\to\OP{Diff}(M)\]
is homotopic to the constant loop $t\mapsto\OP{id}.$ Let 
\[\Phi^s:\mathbb{R}/\mathbb{Z}\to \OP{Diff}(M),\quad s\in[0,1]\]
be a homotopy between the loops $t\mapsto \OP{id}$ and $t\mapsto\sigma^2_t.$ An isotopy between \eqref{eq:SqDTw} and the identity is given by
\[[0,1]\times\hat{W}\:\ni\:(s,p)\mapsto  \left\{ \begin{matrix}
\left( \Phi^s_{v(r)}(x),r\right)&\text{for }p=(x,r)\in M\times (0,\infty),\\
p&\text{otherwise},
\end{matrix}\right.\]
\end{proof}

Albers and McLean found in \cite{MR2817777} a large family of contact manifolds such that every strong filling of them has infinite dimensional symplectic homology. It is an interesting question, in the view of Corollary~\ref{thm:DTInfOrd}, whether any of their exact fillable examples has a periodic Reeb flow.

\begin{example}
Assume $W$ is a higher-genus closed surface with an open disk removed. Since $H_\ast(W;\mathbb{Z}_2)$ is not symmetric, Corollary~\ref{thm:DTNonTriv} implies that the fibered Dehn twist represents a nontrivial class in $\pi_0\OP{Symp}_c(W,\omega).$ However, the Fibered Dehn twist in this case is the Dehn twist along (a translate of) the boundary circle.
\end{example}

\begin{example}\label{HyperplaneSection}
Let $V$ be a smooth degree $d\geqslant 2$ projective hypersurface in $\mathbb{C}P^m,\:m>3,$ and let $H$ be a hyperplane transverse to $V.$  By removing a neighbourhood of $V\cap H$ in $V,$ we get a Liouville domain $W$ such that the Reeb flow induces a free circle action on its boundary $\partial W.$ By \cite{MR894295}, the manifold $W$ has the homotopy type of a bouquet of $(m-1)$-spheres. Moreover, the number of the spheres is equal to $(d-1)^m$ (see Theorem~1 and consequence (iii) on page~487 in \cite{MR2018927}). Hence $c_1(W)=0$ and the homology $H_\ast(W;\mathbb{Z}_2)$ is not symmetric. Corollary\ref{cor:DTNonTrivC0} implies that the fibered Dehn twist of $W$ is not symplectically isotopic to the identity relative to the boundary. (See also Example~7.14 in~\cite{ChiangDingKoert2014}.) 
\end{example}
\section{Generalization and further applications}
\subsection{Generalized slope}
In this section we prove the maximum principle which enables us to extend the definition of Floer homology to a wider class of triples $(\phi,H,J).$ More precisely, we replace condition~\eqref{HamLin} by the following weaker condition
\begin{equation}
H_t(x,r)= rh_t(x),\quad\text{near infinity,}
\end{equation}
where $h:\mathbb{R}\times M\to (0,\infty)$ is a (time-dependent) contact Hamiltonian such that $\sigma^h_1$ is a strict contactomorphism, i.e. a contactomorphism $M\to M$ that preserves the contact 1-form $\beta$ (see Definition~\ref{def:ContHam} below).
A similar maximum principle has been independently proven by Ritter \cite[Appendix~C]{Ritter}.
\begin{definition}\label{def:ContHam}
A \textbf{contact Hamiltonian} $h$ on $M$ is a smooth function $h:M\to \mathbb{R}.$ We can associate to it a vector field $Y_h$ on M in the following way. Consider the Hamiltonian $H:M\times (0,\infty)\to\mathbb{R}:(x,r)\mapsto rh(x)$  on the symplectization. We define $Y_h:=d(\OP{pr}_M)X_h,$ where $\OP{pr}_M$ is the projection. A time-dependent contact Hamiltonian $h:\mathbb{R}\times M\to \mathbb{R}$ generates the contact isotopy $\sigma^h_t:M\to M$ defined by $\partial_t\sigma_t^h=Y_{h_t}\circ\sigma^h_t,\:\sigma^h_0=\OP{id}.$
Here, $h_t:=h(t,\cdot).$ We denote by $f^h_t$ the positive function $M\to (0,\infty)$ defined by $(\sigma^h_t)^\ast\beta=f_t^h\beta.$
\end{definition}
\begin{remark}\label{rmk:hY}
Let $h:M\to \mathbb{R}$ be a contact Hamiltonian. Then
$h_t=-\beta(Y_{h_t}).$
\end{remark}

\begin{definition}
A contact Hamiltonian $h_t:M\to \mathbb{R}$ is called \textbf{admissible} if it is 1-periodic and the map $\sigma^h_1:M\to M$ is a strict contactomorphism without fixed points.
\end{definition}
\begin{definition}\label{def:FloerDataGen}
Let $\phi\in\OP{Symp}_c(W,\lambda/d),$ and let $h$ be an admissible contact Hamiltonian. \textbf{Floer data} for $(\phi,h)$ consists of a (time-dependent) Hamiltonian $H_t:\hat{W}\to\mathbb{R}$ and a family $J_t$ of $\omega$-compatible almost complex structures on $\hat{W}$ satisfying the following conditions. $H$ and $J$ are twisted by $\phi,$ i.e.
\begin{align*}
H_{t+1}&=H_t\circ\phi,\\
J_{t+1}&=\phi^\ast J_t.
\end{align*}
In addition, $H_t(x,r)=rh_t(x)$ near infinity and the almost complex structure $(\psi_t^H)^\ast J_t$ satisfies conditions $\eqref{JCondInfComp}$ and $\eqref{JCondInfDel}$ near infinity. We will refer to the (time-dependent) contact Hamiltonian $h$ as the \textbf{slope} of $H.$
\end{definition}
Regular Floer data for $(\phi,h)$ is defined in analogy to Definition~\ref{def:RegFlDat}.
\begin{lemma}\label{thm:MaxPrincGen}
Let $h_t:M\to\mathbb{R}$ be a contact Hamiltonian, let $H_t:=rh_t: M\times(0,\infty)\to\mathbb{R},$ let $J_t$ be a family of almost complex structures on $M\times(0,\infty)$ such that $(\psi_t^H)^\ast J_t$ satisfies conditions \eqref{JCondInfComp} and \eqref{JCondInfDel}, and let
\[u=(v,g):U\to M\times(0,\infty)\]
be a solution of the Floer equation \eqref{FloerEq}, where $U$ is an open connected subset of $\mathbb{R}^2.$ Then the function 
\[(s,t)\mapsto g(s,t)f^h_t\left( (\sigma^h_t)^{-1}\circ v(s,t)\right)\]
has no local maxima in $U$ unless it is constant.
\end{lemma}
\begin{proof}
Note that, under conditions of the lemma, $(s,t)\mapsto (\psi^H_t)^{-1}\circ u(s,t)$ is a holomorphic curve with respect to the family of almost complex structures $(\psi^H_t)^\ast J_t.$ Lemma~\ref{thm:maxprinc} below implies the maximum principle for
\[r\circ(\psi^H_t)^{-1}\circ u=gf^h_t\left( (\sigma^h_t)^{-1}\circ v\right)\]
Hence, the proof is finished.
\end{proof}

Let $\phi$ and $h$ be as in Definition~\ref{def:FloerDataGen}, and let $(H,J)$ be regular Floer data for $(h,a).$ Using Lemma~\ref{thm:MaxPrincGen}, we can prove Theorems~\ref{Moduli}, \ref{Moduli1} and \ref{Moduli2} in this setting.

The groups $\{HF_\ast(\varphi,H,J)\}_{(H,J)}$ with slope of $H$ equal to $h$ are naturally isomorphic to each other (via the continuation maps). Hence, the group
\[HF_\ast(\phi,h)\]
is well defined. This is consistent with the theory developed in Section~\ref{sec:HF}, i.e. if we take $h$ to be constant, we get the group considered in Section~\ref{sec:HF}.

The existence of the continuation morphisms is more delicate than in Section~\ref{sec:HF}. We are not able to construct the continuation map
\begin{equation}\label{eq:contmapgen}
HF_\ast(\phi,h)\to HF_\ast(\phi,h')
\end{equation}
for two admissible contact Hamiltonians
\[h_t,h_t': M\to \mathbb{R}\]
only under the condition
\[h_t(x)\leqslant h'_t(x)\]
for all $t\in\mathbb{R}$ and $x\in M.$ However, if we assume additionally that the contact Hamiltonians $h_t$ and $h_t'$ generate families of strict contactomorphisms, i.e.
\[dh_t(R^\beta)=dh'_t(R^\beta)=0,\]
Lemma~\ref{thm:maxprinc} below implies that the continuation map \eqref{eq:contmapgen} exists.

\begin{lemma}\label{thm:maxprinc}
Let $h_{s,t}:M\to\mathbb{R}$ be a family of contact Hamiltonians ($s,t\in\mathbb{R}$) such that $dh_{s,t}(R^\beta)=0,$ let $H_{s,t}:M\times (0,\infty)\to \mathbb{R}$ be a Hamiltonian and $J_{s,t}$ a family of almost complex structures on $M.$ Assume that \[H_{s,t}(x,r)=rh_{s,t}(x)\] 
and that $J_s$ satisfies conditions~\eqref{JCondInfComp} and \eqref{JCondInfDel} for each $s\in\mathbb{R}.$ Let $U$ be an open connected subset of $\mathbb{R}^2.$ If the map 
\[u=(v,g)\::\:U\to M\times(0,\infty)\]
is a solution of the Floer equation
\[\partial_s u + J_{s,t}(u) \left( \partial_t u - X_{H_{s,t}}(u)\right)=0\]
and if 
\begin{equation}\label{IncSlopeEq}
\partial_sh_{s,t}\geqslant 0, 
\end{equation} 
then $g$ has no local maxima in $U$ unless it is constant.
\end{lemma} 
\begin{proof}
We mimic the proof of Proposition~4.1 in \cite{BegOverview-Wendl}. For simplicity, we denote $Y_{h_{s,t}}$ by $Y_{s,t}.$ The condition $dh_{s,t}(R^\beta)=0$ is equivalent to
\begin{equation}
\label{CondYEq}d\left(\beta(Y_{s,t}) \right) +d\beta(Y_{s,t},\cdot)=0.
\end{equation}
It is not difficult to see that \[X_{H_{s,t}}(x,r)=Y_{s,t}(x)\]
for $(x,r)\in M\times (r_0,\infty).$ Using the Floer equation and the splitting $\left\langle \partial_r \right\rangle \oplus \left\langle R^\beta\right\rangle \oplus \xi$ of $T\left(M\times (r_0,\infty)\right),$ we get that $v$ and $g$ satisfy the following system of partial differential equations.
\begin{align}
0&=  \partial_s g-\beta(\partial_t v)+\beta(Y_{s,t})\circ v, \label{MaxPrinEq1}\\
0&=\beta(\partial_s v)+\partial_t g,\label{MaxPrinEq2}\\
0&= \pi_\beta \partial_s v+ J_{s,t}\circ u\: \pi_\beta \partial_t v - J_{s,t}\circ u\: \pi_\beta Y_{s,t}\circ v, \label{MaxPrinEq3}
\end{align}
where $\pi_\beta$ stands for the projection $TM\to \xi^\beta$ (along the Reeb vector field).
Remark~\ref{rmk:Compatibility} implies
\begin{equation}
\begin{split}
0&\leqslant d\beta \left( \pi_\beta \partial _s v, J_{s,t}\circ u\:\pi_\beta\partial_s v \right)\\
&\overset{\eqref{MaxPrinEq3}}{=\joinrel=\joinrel=} d\beta\left(\pi_\beta \partial_s v, \pi_\beta\partial_t v \right) - d\beta\left( \pi_\beta \partial_s v, \pi_\beta Y_{s,t}\circ v\right)\\
&= d\beta(\partial_s v,\partial_t v) - d\beta\left(\partial_s v, Y_{s,t}\circ v\right)\\
&= \partial_s\left( \beta(\partial_t v)\right) -\partial_t \left(\beta (\partial_s v)\right) - d\beta\left(\partial_s v, Y_{s,t}\circ v\right)\\
&\overset{\eqref{MaxPrinEq1}\&\eqref{MaxPrinEq2}}{=\joinrel=\joinrel=\joinrel=\joinrel=\joinrel=\joinrel=} \partial_s\left( \partial_s g+\beta(Y_{s,t})\circ v \right) - \partial_t(-\partial_t g) - d\beta(\partial_s v, Y_{s,t}\circ v)\\
&= \Delta g + \partial_s(\beta(Y_{s,t}))\circ v+ d(\beta (Y_{s,t}))(\partial_s v) + d\beta (Y_{s,t}\circ v, \partial_s v)\\
&\overset{\eqref{CondYEq}}{=\joinrel=\joinrel=} \Delta g + \partial_s(\beta(Y_{s,t}))\circ v\\
&=\Delta g - \partial_sh_{s,t}\circ v.
\end{split}
\end{equation}
This together with \eqref{IncSlopeEq} yields $\Delta g\geqslant 0.$
\end{proof}
\begin{proof}[Proof of Lemma~\ref{thm:maxprinclin}]\label{pf:maxprinclin}
The lemma is a special case of Lemma~\ref{thm:maxprinc} for $h_{s,t}(x)= a(s,t).$
\end{proof}

\begin{proposition}\label{prop:SlopePerGen}
Let $\phi\in\OP{Symp}_c(W,\lambda/d),$ let $h_t:M\to \mathbb{R}$ be an admissible contact Hamiltonian, and let $ K_t:\hat{W}\to\mathbb{R}$ be a Hamiltonian such that $K_t\circ \phi= K_{t+1},$ such that $K_t(x,r)=rk_t(x)$ near infinity for a (1-periodic) contact Hamiltonian $k_t:M\to \mathbb{R},$ and such that $\psi^K_1\in\OP{Symp}_c(W,\lambda/d).$ Then there exists an isomorphism
\[\N(K)\::\: HF(\phi,h)\to HF(\phi\circ \psi^K_1,\tilde{h}),\]
where $\tilde{h}_t:M\to\mathbb{R}$ is defined by
\[\tilde{h}_t(x)=\frac{h_t(\sigma^k_t(x))-k_t(\sigma^k_t(x))}{f^k_t(x)}.\]
\end{proposition}
\begin{proof}
The proof is analogous to the proof of Proposition~\ref{prop:nat}. Note that the Hamiltonian $\left(\psi^K\right)^\ast (H-K)$ has the slope equal to $\tilde{h}_t,$ where $H_t:\hat{W}\to\mathbb{R}$ is a Hamiltonian with the slope $h_t.$
\end{proof}
The groups $HF(\phi, h)$ are invariant (up to isomorphism) under compactly supported symplectic isotopies of $\phi.$
\subsection{A homotopy long exact sequence}
In this section, we describe a long exact sequence due to Giroux \cite{BiranGiroux}, which provides further examples of elements in $\OP{Symp}_c(W,\omega).$ The fibered Dehn twist can be seen as a special case of this construction (see Proposition~\ref{prop:DTandLongES} below).

\begin{lemma}\label{thm:corr}
Let $\phi\in\OP{Symp}(\hat{W},\lambda/d).$ Then there is a contactomorphism $\phi_M:M\to M$ and a function $f:M\to (0,\infty)$ such that
\[\phi_M^\ast\beta=f\beta\]
and
\[\phi(x,r)=\left(\phi_M(x),\frac{r}{f(x)}\right)\]
for all $x\in M$ and all $r\in (0,\infty)$ large enough.
\end{lemma}
\begin{proof}
There exists $r_0\in (0,\infty)$ such that the restriction of $\phi$ to $M\times (r_0,\infty)$ is a $\lambda$-preserving embedding $M\times (r_0,\infty)\to M\times(0,\infty).$ Let 
\[\phi(x,r)=(\sigma(x,r),g(x,r))\in M\times(0,\infty)\]
for $(x,r)\in M\times(r_0,\infty).$ Since $\phi$ preserves $\lambda$ and $\omega,$ it preserves the vector field $X_\lambda=r\partial_r,$ because $X_\lambda$ is characterized by the relation $X_\lambda\lrcorner \omega=\lambda.$ Thus, $\phi$ commutes with the flow of $r\partial_r,$ i.e. 
\[\left(\sigma\left(x,e^tr\right),g\left(x,e^tr\right)\right)= \left(\sigma\left(x,r\right),e^tg\left(x,r\right)\right).\]
This implies that $\sigma(x,r)$ does not depend on $r.$ The map $\sigma(\cdot,r),$ denoted by $\phi_M,$ is a contactomorphism, because
\[r\beta=\phi^\ast(r\beta)= g(x,r)\phi_M^\ast \beta\] 
implies 
\[\phi^\ast_M=\frac{r}{g(x,r)}\beta\]
and $\frac{r}{g(x,r)}$ is a positive function. Another consequence is that $\frac{r}{g(x,r)}=:f(x)$ does not depend on $r.$ Hence, the proof is finished.
\end{proof}

\begin{definition}\label{def:irm}
In the view of Lemma~\ref{thm:corr}, one can associate a contactomorphism ${\phi_M: M\to M}$ to every exact symplectomorphism $\phi:\hat{W}\to \hat{W}.$ This gives rise to the homomorphism 
\[\Theta\::\:\OP{Symp}\left(\hat{W},\lambda/d\right)\to \OP{Cont}\left(M,\xi^\beta\right)\::\: \phi\mapsto \phi_M,\]
which we call the \textbf{ideal restriction map}.
\end{definition}

Let $\OP{Symp}(W,\lambda/d)$ be the group of symplectomorphisms \[\varphi:\widehat{W}\to \widehat{W}\]
such that $\varphi^\ast\lambda-\lambda= dF$ for a function $F:\widehat{W}\to \mathbb{R}$ compactly supported in $W-M.$ It turns out that the ideal restriction map of Definition~\ref{def:irm}
\[\Theta\::\:\OP{Symp}(W,\lambda/d)\to \OP{Cont}(M,\xi^\beta)\]
is a Serre fibration \cite{BiranGiroux} with fibre over the identity equal to 
\[\OP{Symp}_c(W,\lambda/d):=\left\lbrace \varphi:W\to W\:|\:\varphi^\ast\lambda-\lambda \text{ is exact }\& \OP{supp} \varphi\in W-M \right\rbrace.\]
Hence, there is a homotopy long exact sequence 
\begin{equation}\label{LESEq}
\begin{tikzcd}[column sep=small]
\cdots \arrow{r} & \pi_k\OP{Symp}_c(W,\lambda/d) \arrow{r}
& \pi_k\OP{Symp}(W,\lambda/d) \arrow{r}
\arrow[draw=none]{d}[name=Z, shape=coordinate]{}
& \pi_k\OP{Cont}(M,\xi^\beta) \arrow[rounded corners,
to path={ -- ([xshift=2ex]\tikztostart.east)
|- (Z) [near end]\tikztonodes
-| ([xshift=-2ex]\tikztotarget.west)
-- (\tikztotarget)}]
{dll}{\Delta} \\
 \arrow[draw=none]{r} & \pi_{k-1}\OP{Symp}_c(W,\lambda/d)\arrow{r}
& \cdots.
\end{tikzcd}
\end{equation}
In particular, there is a connecting homomorphism
\begin{equation}\label{DeltaBGEq}
\Delta\::\:\pi_1\OP{Cont}(M,\xi^\beta) \to \pi_0\OP{Symp}_c(W,\lambda/d).
\end{equation} 
In view of Lemma~\ref{thm:HomEquiv}, the inclusion 
\[\OP{Symp}_c(W,\lambda/d)\to \OP{Symp}_c(W,\omega)\]
induces isomorphism in homotopy groups, so that the groups $\pi_i\OP{Symp}_c(W,\lambda/d)$ in \eqref{LESEq} and \eqref{DeltaBGEq} can be replaced by $\pi_i\OP{Symp}_c(W,\omega).$

\begin{proposition}\label{prop:DTandLongES}
Assume the Reeb flow $\sigma^\beta_t$ is 1-periodic. Then the image $\Delta([\sigma])$ of the class $[\sigma]\in\pi_1\OP{Cont}(M,\xi^\beta)$ under the connecting homomorphism $\Delta$ is represented by the fibered Dehn twist.
\end{proposition}

\begin{lemma}\label{thm:ImConnTimeOne}
Let $\sigma_t: M\to M$ be a 1-periodic family of  contactomorphisms generated by a vector field $Y_t.$ Then the family $\{\sigma_t\}$ determines an element of $\pi_1 \OP{Cont}(M,\xi^\beta)$ whose image under the connecting homomorphism $\Delta$ can be represented by the time-1 map of a (non-compactly supported) Hamiltonian $\hat{W}\to\mathbb{R}$ that is equal $-r\beta(Y_t)$ near infinity. 
\end{lemma}
\begin{proof}
The element in $\pi_1\OP{Cont}(M,\xi^\beta)$ determined by $\{\sigma_t\}$ is denoted (by a slight abuse of notation) by $[\sigma].$ By definition, $\Delta([\sigma])$ is equal to $[\psi_1]\in \pi_0\OP{Symp}_c(W,\omega),$ where $\psi_t$ is a path in $\OP{Symp}(W,\lambda/d)$ such that $\Theta(\psi_t)=\sigma_t$ (see Definition~\ref{def:irm}). Let $\chi :(0,\infty)\to \mathbb{R}$ be a smooth function equal 0 near 0 and 1 on $[1,\infty).$ We can choose $\psi_t$ from above to be $\psi^H_t$ where $H:\hat{W}\to \mathbb{R}$ is the Hamiltonian that is equal to
\[(x,r)\mapsto -\chi(r)r\beta(Y_t)\]
on $M\times(0,\infty)$ and $0$ elsewhere. This finishes the proof.
\end{proof}
\begin{lemma}\label{thm:SameNearInf}
Let $H_t$ and $K_t$ be two Hamiltonians on $\hat{W}$ that are equal near infinity and whose isotopies are well-defined for all times. If the maps  $\psi^H_1$ and $\psi_1^K$ are compactly supported in $W-M,$ then they represent the same class in $\pi_0\OP{Symp}_c(W,\omega).$
\end{lemma}
\begin{proof}
The Hamiltonian $H\#\overline{K}$ (see page~\pageref{notation:sharp}) is compactly supported and generates the Hamiltonian isotopy $\psi_t^H\circ\left(\psi_t^K\right)^{-1}$ (compactly supported as well). This isotopy, however, may not be supported in the interior of $W$ for all times. Let $\phi^\lambda_t$ be the family of diffeomorphisms of $\hat{W}$ that is generated by the Liouville vector field $X_\lambda$ (see section~\ref{sec:Conv}). The diffeomorphism $\phi^\lambda_t$ is not a symplectomorphism (except for $s=0$), however if we conjugate a symplectomorphism by it we get a symplectomorphism. Since the isotopy $\psi_t^H\circ\left(\psi_t^K\right)^{-1}$ is compactly supported, there exists $s\in (0,\infty)$ such that $\left(\phi^\lambda_s\right)^{-1}\circ\psi^H_t\circ\left(\psi^K\right)^{-1}\circ \phi^\lambda_s$ is compactly supported in the interior of $W.$ In particular, the maps $\left(\phi^\lambda_s\right)^{-1}\circ \psi^H_1\circ\phi^\lambda_s$ and $\left(\phi^\lambda_s\right)^{-1}\circ \psi^K_1\circ\phi^\lambda_s$ represent the same class in $\pi_0\OP{Symp}_c(W,\omega).$ Note that $\left(\phi^\lambda_t\right)^{-1}\circ \psi^H_1\circ\phi^\lambda_t\in \OP{Symp}_c(W,\omega)$ for all $t\in (0,\infty).$ Hence $\left(\phi^\lambda_s\right)^{-1}\circ \psi^H_1\circ\phi^\lambda_s$ and $\psi^H_1$ represent the same class in $\pi_0\OP{Symp}_c(W,\omega).$ Similarly, the same holds for $\left(\phi^\lambda_s\right)^{-1}\circ \psi^K_1\circ\phi^\lambda_s$ and $\psi^K_1,$ and the proof is finished.
\end{proof}
\begin{proof}[Proof of Proposition~\ref{prop:DTandLongES}]
The proposition follows from Lemma~\ref{thm:ImConnTimeOne} (applied in the case $Y_t=R^\beta$), Lemma~\ref{thm:SameNearInf} and Remark~\ref{HamGenFib}.
\end{proof}

\begin{theorem}\label{thm:SlopePeriodicityGen}
Let $\sigma_t:M\to M$ be a 1-periodic family of contactomorphisms generated by a vector field $Y_t,$ and let $f_t:M\to(0,\infty)$ be the function defined by $\sigma^\ast\beta=f_t\beta.$ Denote by $[\sigma]$ the element of $\pi_1\OP{Cont}(M,\xi^\beta)$ determined by $\{\sigma_t\}.$ If the image $\Delta([\sigma])\in \OP{Symp}_c(W,\lambda/d)$ of $[\sigma]$ under the connecting homomorphism $\Delta$ is trivial, then
\[HF(\OP{id},h)\cong HF(\OP{id},\tilde{h}).\]
Here $h_t:M\to\mathbb{R}$ is an admissible contact Hamiltonian and $ \tilde{h}:M\to \mathbb{R}$ is defined by
\[\tilde{h}_t(x)=\frac{h_t(\sigma_t(x))+\beta(Y_t(x))}{f_t(x)},\quad x\in M.\]
\begin{proof}
The proof follows from Proposition~\ref{prop:SlopePerGen}, Lemma~\ref{thm:ImConnTimeOne} and the invariance under compactly supported symplectic isotopies.
\end{proof}
\end{theorem}

\subsection{Questions on translated points}
Here we use our theory to study a conjecture due to Sandon. Most of the results in the section can be also proved by other methods and are probably known to experts \cite{MR2890476,MR3079344,MR3071949,AlbersFuchsMerry,AlbersMerry,Zenaidi}.
\begin{definition}(Sandon \cite{MR2890476})
Let $\sigma:M\to M$ be a contactomorphism. The point $x\in M$ is said to be a \textbf{translated point} of the contactomorphism $\sigma$ with respect to the contact form $\beta$ if $(\sigma^\ast \beta)_x=\beta_x$ and if $x$ and $\sigma(x)$ lie on the same Reeb orbit.
\end{definition}
\begin{conjecture}(Sandon \cite{MR3079344})
Let $\sigma :M\to M$ be a contactomorphism that is contact isotopic to the identity. Then the number of translated points is greater than or equal to the minimal number of critical points a smooth function on $M$ may have.
\end{conjecture}
\begin{definition}
Let $x\in M$ be a translated point of a contactomorphism $\sigma:M\to M.$ The number $a\in \mathbb{R}$ is called a \textbf{period} of $x$ if $\sigma^\beta_a(x)=\sigma(x).$
\end{definition}
\begin{remark}
Note that a translated point may have many periods. Moreover, a period does not have to be non-negative.
\end{remark}
\begin{lemma}[Giroux]\label{thm:ContImTheta}
Let $\sigma:M\to M$ be a contactomorphism isotopic to the identity through contactomorphisms. Then $\sigma\in \OP{im}\Theta$ (see Definition~\ref{def:irm}).
\end{lemma}
\begin{proof}
Let $\sigma_t:M\to M,\:t\in [0,1]$ be a family of contactomorphisms such that $\sigma_0=\OP{id}$ and $\sigma_1=\sigma.$ The contact isotopy $\sigma_t$ gives rise to a Hamiltonian isotopy $M\times(0,\infty)\to M\times(0,\infty)$ of the symplectization. The Hamiltonian is given by 
\[(x,r)\mapsto -r\beta(Y_t(x)),\] 
where $Y_t$ is the vector field of $\sigma_t.$ Let $\chi :(0,\infty)\to \mathbb{R}$ be a smooth function equal 0 near 0 and 1 on $[1,\infty).$ Consider the Hamiltonian $H_t:\hat{W}\to \mathbb{R}$ defined by
\[H_t(x,r):=-\chi(r)r\beta(Y_t(x))\]
for $(x,r)\in M\times(0,\infty)$ and $H_t\equiv 0$ elsewhere. The Hamiltonian isotopy $\psi^H_t:\hat{W}\to \hat{W}$ consists of exact symplectomorphisms such that $\Theta(\psi_t^H)=\sigma_t.$ In particular, $\Theta(\psi_1^H)=\sigma,$ and the lemma is proved.
\end{proof}

The following theorem follows from the results in \cite{MR3071949}.
\begin{theorem}\label{thm:AppSan}
Let $\sigma:M\to M$ be a strict contactomorphism that is isotopic to the identity through strict contactomorphisms. If 
\begin{equation}\label{eq:dimSHInf}
\dim SH(W;\mathbb{Z}_2)=\infty,
\end{equation}
then there are infinitely many translated points of $\sigma.$
\end{theorem}
\begin{proof}
A strict contactomorphism commutes with the Reeb flow. Hence, if $\sigma$ has a translated point, then it has infinitely many of them (all point on a Reeb orbit containing a translated point are translated points).

Without loss of generality, we may assume $\sigma=\sigma_1,$ where $\sigma_t:M\to M$ is a family of strict contactomorphisms with $\sigma_0=\OP{id}$ that is generated by 1-periodic vector field $Y_t$ on $M.$ Let $H_t:\hat{W}\to\mathbb{R}$ be a Hamiltonian such that $\Theta(\psi^H_t)=\sigma_t.$ Its existence follows from the proof of Lemma~\ref{thm:ContImTheta}. Moreover, $H_t$ is equal to $rh_t$ near infinity, where \[h_t:M\to\mathbb{R}:x\mapsto -\beta(Y_t(x)).\]
The Hamiltonian isotopy  of the Hamiltonian 
\[M\times (0,\infty)\ni (x,r)\mapsto r(h_t(x)+a)\]
is equal to $(x,r,t)\mapsto (\sigma_t\circ\sigma^\beta_{-at}(x),r).$ Therefore, 1-periodic orbits of this Hamiltonian correspond to translated points of $\sigma$ that have period $a.$ Assume that $\sigma$ has no translated points. Then, similarly as in the proof of Lemma~\ref{thm:VarSl}, one can prove that the continuation map
\[HF_\ast(\OP{id},h+a)\to HF_\ast(\OP{id},h+a'),\quad, a\leqslant a'\]
is an isomorphism for $a$ large enough. Hence
\[\underset{a\in\mathbb{R}}{\lim\limits_{\longrightarrow}}HF(\OP{id},h+a)\cong HF(\OP{id},h+a_0)\]
for $a_0\in\mathbb{R}$ large enough. In particular,
\[\dim \underset{a\in\mathbb{R}}{\lim\limits_{\longrightarrow}}HF(\OP{id},h+a)< \infty. \]
On the other hand, for each $a'\in\mathbb{R},$ there exist $a,a''\in\mathbb{R}$ such that $a'\leqslant h(x)+a\leqslant a'',$ for all $x\in M.$ Therefore  
\[\underset{a\in\mathbb{R}}{\lim\limits_{\longrightarrow}}HF(\OP{id},h+a)\cong \underset{a\in\mathbb{R}}{\lim\limits_{\longrightarrow}}HF(\OP{id},a)=: SH(W;\mathbb{Z}_2),\]
and the proof is finished.
\end{proof}

\printbibliography 
\end{document}